\documentclass{article}

\usepackage{geometry}
\usepackage[mathcal]{euscript}
\usepackage{hyperref}
\usepackage{amsmath,amssymb,amsthm}
\usepackage{mathtools}
\usepackage{dutchcal}
\usepackage{newtxtext,newtxmath,helvet,courier}
\usepackage{fix-cm}
\DeclareMathAlphabet{\mathpzc}{OT1}{pzc}{m}{it}

\usepackage{xcolor}
\usepackage{cases}
\newtheorem{thm}{Theorem}[section]
\newtheorem{lemma}[thm]{Lemma}
\newtheorem{prop}[thm]{Proposition}
\newtheorem{cor}[thm]{Corollary}
\theoremstyle{remark}
\newtheorem{rem}[thm]{Remark}

\newtheorem{conj}[thm]{Conjecture}
\theoremstyle{definition}
\newtheorem{defn}[thm]{Definition}
\newtheoremstyle{Claim}{}{}{\itshape}{}{\itshape\bfseries}{:}{ }{#1}
\theoremstyle{Claim}

\newcommand{\Z}{{\mathbb{Z}}}

\newcommand{\T}{{\mathbb{T}^d}}

\renewcommand{\H}{\mathcal{H}}

\newcommand{\R}{\mathbb{R}}

\newcommand{\N}{\mathbb{N}}

\newcommand\sP{\mathpzc{P}}
\newcommand\sQ{\mathpzc{Q}}

\newcommand{\eps}{\varepsilon}
\newcommand{\norm}[1]{\left\lVert#1\right\rVert}
\DeclareMathOperator{\dive}{div}

\def\ds{\displaystyle}

\titlepage
\title{Transport equations with nonlocal diffusion and applications to Hamilton-Jacobi equations}
\author{Alessandro Goffi}
\date{\today}

\begin{document}

\maketitle

\begin{abstract} 
We investigate regularity and a priori estimates for Fokker-Planck and Hamilton-Jacobi equations with unbounded ingredients driven by the fractional Laplacian of order $s\in(1/2,1)$. As for Fokker-Planck equations, we establish integrability estimates under a fractional version of the Aronson-Serrin interpolated condition on the velocity field and Bessel regularity when the drift has low Lebesgue integrability with respect to the solution itself. Using these estimates, through the Evans' nonlinear adjoint method we prove new integral, sup-norm and H\"older estimates for weak and strong solutions to fractional Hamilton-Jacobi equations with unbounded right-hand side and polynomial growth in the gradient. Finally, by means of these latter results, exploiting Calder\'on-Zygmund-type regularity for linear nonlocal PDEs and fractional Gagliardo-Nirenberg inequalities, we deduce optimal $L^q$-regularity for fractional Hamilton-Jacobi equations.
\end{abstract}

\noindent
{\footnotesize \textbf{AMS-Subject Classification}}. {\footnotesize 35R11, 35F21, 35Q84, 35B65.}

 \tableofcontents
 
 \section{Introduction}
In this paper, we analyze regularity properties of transport equations of Fokker-Planck-type and Hamilton-Jacobi equations with fractional diffusion driven by a fractional power of the Laplacian, $(-\Delta)^s$, with subcritical order $s\in(\frac12,1)$. In particular, we address well-posedness, parabolic Bessel regularity and integrability estimates for solutions to (backward) fractional Fokker-Planck equations of the form
\begin{equation}\label{fp}
\begin{cases}
-\partial_t \rho(x,t)+(-\Delta)^s\rho(x,t)+\dive(b(x,t)\, \rho(x,t))=0&\text{ in }Q_\tau:=\T\times(0,\tau)\ ,\\
\rho(x,\tau)=\rho_\tau(x)&\text{ in }\T\ ,
\end{cases}
\end{equation}
where the nonlocal diffusion operator is defined on the flat torus $\T\equiv\R^d\backslash\Z^d$ \cite{RS}, under ``rough'' integrability conditions on the velocity field, mainly when either $b\in L^\sQ_t(L^\sP_x)$ (cf \eqref{fractionalAS} below) or $b\in L^k(\rho\,dxdt)$, $k>1$, without requiring a control on its divergence. \\
Our second aim is to apply the  results for the above transport-diffusion equation to obtain a priori gradient estimates for strong solutions and regularization effects for weak solutions of fractional Hamilton-Jacobi equations with subcritical diffusion of the form
\begin{equation}\label{hjb}
\begin{cases}
\ds \partial_tu(x,t) +(-\Delta)^s u(x,t) + H(x, Du(x,t)) = f(x,t) & \text{in $Q_T$}, \\
u(x, 0) = u_0(x) & \text{in $\T$,}
\end{cases}
\end{equation}
where $f\in L^q(Q_T)$ for some $q>1$ and $H(x,Du)\sim |Du|^\gamma$, $\gamma>1$, i.e. $H$ has superlinear gradient growth. \\

Following the approach in \cite{CG2,CG5}, we first obtain Sobolev-type regularity for solutions to \eqref{fp}. This level of regularity is crucial to derive new integral, sup-norm and H\"older estimates for solutions to \eqref{hjb} by means of the nonlinear adjoint method introduced by L.C. Evans \cite{Evans,EvansCVPDE}. These results are then combined  with Gagliardo-Nirenberg interpolation inequalities and maximal regularity in Lebesgue spaces for fractional heat equations to obtain optimal regularity in Lebesgue spaces for \eqref{hjb}. This approach to deduce a priori estimates for nonlinear problems has been inspired by \cite{AC} (see also \cite{Bens1,Bens2} for later contributions), where semi-linear equations with quadratic growth in the gradient have been studied. These interpolation methods have been also employed in e.g. \cite{S2} (see also the references therein) and recently revived in \cite{GomesBook,Gomessub,CG5} in the context of Mean Field Games \cite{ll,LLtime}. In particular, our results extend those in \cite{CG2,CG5} in the fractional framework for $s\in(1/2,1)$, both for \eqref{fp} and \eqref{hjb}. \\

As announced, in order to study the regularity properties of \eqref{fp} we need to extend well-known results for linear viscous equations with unbounded coefficients to the fractional framework. In the viscous case, the first works date back to \cite{LSU,Aronson,AS} for linear and quasi-linear problems, see also \cite{BOP} for the case of measurable ingredients. Within this framework, well-posedness and integrability estimates are well-established when $b\in L^\sQ_t(L^\sP_x)$ with $\sQ,\sP$ satisfying 
\[
\frac{d}{2\sP}+\frac1\sQ\leq\frac12\ ,\sP\in[d,\infty]\ ,\sQ\in[2,\infty]\ ,
\]
the so-called Aronson-Serrin interpolated condition. We emphasize that this assumption on the drift has been shown to be sharp to get integrability estimates in Lebesgue classes, at least in the viscous case, cf \cite{BCCS}. \\
As far as we know, nonlocal heat equations with unbounded coefficients have been treated in \cite{Leonori}, see also \cite{JS} for gradient perturbations of the fractional Laplacian and the nonlinear analysis carried out in \cite{AP,PPS} in the context of (fractional) Kardar-Parisi-Zhang models. The well-posedness and integrability results we present here are new when the velocity field satisfies a fractional version of the above Aronson-Serrin condition, i.e. $b\in L^\sQ_t(L^\sP_x)$ with $\sQ,\sP$ fulfilling
\begin{equation}\label{fractionalAS}
\frac{d}{2s\sP}+\frac1\sQ<\frac{2s-1}{2s}\ .
\end{equation}
In particular, we mention that in this setting we are not able to cover the equality in \eqref{fractionalAS} neither for the well-posedness nor for integrability estimates of solutions of \eqref{fp}, and this remains at this stage an open problem.\\
The second step in the analysis of \eqref{fp} concerns fractional Bessel regularity estimates of solutions to \eqref{fp} when $b\in L^k(\rho\,dxdt)$ for some $k>1$, i.e. in terms of the crossed term
\begin{equation}\label{cross}
\iint_{Q_\tau}|b|^k\rho\,dxdt\ ,k>1\ ,
\end{equation}
which is widely analyzed for classical viscous Fokker-Planck equations in \cite{MPR,Porr}. In particular, we prove that for some suitable $k>1$ one has the estimate
\[
\|(-\Delta)^{s-\frac12}\rho\|_{\sigma'}\lesssim \iint_{Q_\tau}|b|^k\rho+\|\rho_\tau\|_{p'}
\]
for $\sigma'$ in some range determined in terms of the regularity $p'$ of the terminal data. This bound is obtained by duality, following \cite{MPR,CG2,CG5}, via a maximal regularity estimate for nonlocal equations with divergence-type terms of the form
\[
\|(-\Delta)^{s-\frac12}\rho\|_{\sigma'}\lesssim \|b\rho\|_{\sigma'}+\|\rho_\tau\|_{p'}\ .
\]
These estimates are fundamental to study regularity properties for PDEs arising in Mean Field Games, cf \cite{CT,CG2,Porr,PorrUMI}, see also \cite{CGoffi} for the time-fractional framework. We remark that when $b=-D_pH(x,Du)$, \eqref{fp} becomes the adjoint equation to \eqref{hjb} and bounds on the quantity \eqref{cross} are natural for the mean field equations by duality \cite{Porr,CG1,CG2}.\\

Owing to these results for \eqref{fp}, we deduce sup-norm, integral and H\"older estimates for solutions to \eqref{hjb} with $f\in L^q$. These bounds are obtained by duality and exploit the aforementioned Bessel regularity properties of solutions to Fokker-Planck equations. \\
To our knowledge, these estimates for fractional Hamilton-Jacobi equations have not yet been investigated in the literature, especially in the context of unbounded coefficients in $L^p$ scales. Furthermore, with respect to H\"older estimates, we provide a H\"older's regularization effect when $\gamma\geq 2s$.\\ 
Finally, we are able to partially prove optimal $L^q$-regularity as addressed in \cite{CG4,CG5} for elliptic and evolutive equations respectively, driven by the Laplacian. This means, within our context, a control on $\partial_tu,(-\Delta)^su,|Du|^\gamma\in L^q$ in terms of $f\in L^q$ for an appropriate range of the integrability exponent $q>\bar q$, see Section \ref{sec;conj} for details on the threshold $\bar q$. As a result, we find that \eqref{hjb} behaves in terms of regularity as the fractional heat equation for appropriate values of $q$, despite the presence of the nonlinear coercive term $H=H(x,p)\sim |p|^\gamma$. By letting $s\to1$ we recover the same results of the viscous case, but we produce only partial a priori estimates in the supercritical regime $\gamma>2s$.\\
 In the viscous stationary case, the proof has been given refining the integral Bernstein method, which, however, does not seem the right path to treat both fractional and time-dependent problems like \eqref{hjb}.\\
We recall that maximal $L^q$-regularity properties of Calder\'on-Zygmund-type are well-known for general abstract linear evolution equations, see e.g. \cite{HP,Lamberton,Lunardi}, and are recalled in Lemma \ref{maximal} below, while in the case $s=\frac12$ a result for nonlocal equations with drift terms can be found in \cite{Zhang}.\\

More precisely, our strategy consists in regarding \eqref{hjb} as a perturbation of a fractional heat equation
\[
\partial_t u(x,t)+(-\Delta)^s u(x,t)=f(x,t)-H(x,Du)
\]
where $H(x,Du)\sim |Du|^\gamma$. Then, maximal regularity for linear nonlocal problems, cf \cite{HP}, applied to the above equation yields
\[
\|\partial_t u\|_q+\|(-\Delta)^s u\|_q \lesssim \||Du|^\gamma\|_{q}+\|f\|_q\ .
\]
The second step relies on applying fractional Gagliardo-Nirenberg inequalities involving integral norms of the form
 \[
 \|Du\|_{L^{\gamma q}_{x,t}}^\gamma\lesssim\|u\|_{L^q(H_q^{2s})}^{\gamma\theta}\|u\|_{L^\infty_t(L^r_x)}^{\gamma(1-\theta)}
 \]
 for some $r\in(1,\infty]$, $\theta\in(0,1)$ such that $\theta\gamma<1$ when $\gamma<2s$, and those involving H\"older norms
 \[
 \|Du\|_{L^{\gamma q}_{x,t}}^\gamma\lesssim\|u\|_{L^q(H_q^{2s})}^{\gamma\theta}\|u\|_{L^\infty_t(C_x^\alpha)}^{\gamma(1-\theta)}\ ,\theta<\frac{1}{\gamma}\ ,
 \]
for $\gamma\geq2s$, cf Lemmas \ref{GN1} and \ref{GN2} and \cite{NHolder,Miranda}. This scheme allows to show that maximal $L^q$-regularity for \eqref{hjb} occurs for strong solutions when $f\in L^q(Q_T)$, $q>\frac{d+2s}{(2s-1)\gamma'}$, $\frac{d+6s-2}{d+2s}< \gamma< \frac{s}{1-s}$, i.e. we have
 \[
 \|\partial_t u\|_{L^q(Q_T)}+\| (-\Delta)^su\|_{L^q(Q_T)}+\||Du|^\gamma\|_{L^q(Q_T)}\leq C(\|f\|_{L^q(Q_T)},\|u(0)\|_{W^{2s-2s/q,q}(\T)},q,d,T,s,H)\ .
 \]
We refer to Remark \ref{strongweak} and Remark \ref{restr} for further comments on the restrictions for $\gamma$. At this stage, we do not know neither if our restrictions when $\gamma>2s$ are sharp nor counterexamples to the maximal regularity below the threshold $\frac{d+2s}{(2s-1)\gamma'}$ when $\gamma\leq 2s$. However, by letting $s\to1$ our results agree with those obtained in the local case \cite{CG4,CG5}.\\
We believe that our duality approach to obtain integral and H\"older bounds, together with the maximal regularity results, can be adapted to the stationary counterpart of \eqref{fp} and \eqref{hjb}, leading to new a priori estimates. This will be matter of a future research.\\
Finally, we remark in passing that, in the classical viscous case, the purely quadratic regime $\gamma=2$ can be addressed using the Hopf-Cole transform, cf \cite[Lemma 4.2]{CLLP}. Here, however, when $\gamma=2s$ it is not known whether there exists a fractional analogue of that transformation which allows to reduce \eqref{hjb} into a simpler fractional PDE. Thus, even the natural (critical) growth case becomes not trivial to analyze. \\
 \par\smallskip
 We now recall some related results for \eqref{fp} and \eqref{hjb}. As for fractional Fokker-Planck equations, when $b\in L^\infty$ or some control on the divergence is assumed, we refer to \cite{CCDNV} for stationary problems and to \cite{CG1} for the evolutive case. Instead, the viscous case is well-known, even under weaker assumptions on the velocity field \cite{LSU,MPR,BOP,BCCS,CT,CG2,Porr,CG5}.\\
As for Hamilton-Jacobi equations, H\"older's regularity results have been largely investigated for parabolic problems in the borderline cases $s=0$ and $s=1$. For first order and second order degenerate problems, we refer to \cite{Carda1,CC,CV,CPT}, while we mention \cite{LSU,LU} for the uniformly parabolic case. Recently, H\"older estimates for second order degenerate problems with unbounded right-hand side have been analyzed in \cite{CSil}, see also \cite{SV} for PDEs driven by the Laplacian, via De Giorgi's techniques. H\"older, integral and sup-norm estimates for the parabolic problem have been already addressed in the paper \cite{CG5} for the viscous case $s=1$, and we recover those results by letting $s\to1$. As for integrability estimates, we refer to \cite{CGPT} for the degenerate case and \cite{CG5} for the viscous problem.\\
 H\"older's regularity of fully nonlinear nonlocal equations with super-quadratic first-order terms has been treated in \cite{CardaRainer}, where the regularity stems from the coercivity of $H$ rather than the ellipticity. H\"older's regularization effect of solutions to fractional Hamilton-Jacobi equations with first order terms having at most critical growth $\gamma= 2s$ has been observed by L. Silvestre in \cite{SilvestreHolderHJ}. In this case, the author has also obtained H\"older bounds in the fractional supercritical regime $2s<\gamma<2s+\eps$ imposing some smallness conditions on the data. More recently, a regularization effect when $s=1/2$ has been investigated in \cite{BesovHJ} under a smallness condition on the initial datum in Besov scales. \\
 Instead, the literature on Lipschitz regularity is huge. The conservation of Lipschitz regularity (i.e. with $u(0)\in W^{1,\infty}$) for every $s\in(0,1)$ and a smoothing effect when $s\in(1/2,1)$ go back to \cite{Imbert1} (see also \cite{Imbert2,KW}). Besides, Lipschitz and further regularity for nonlocal Hamilton-Jacobi equations has been investigated in the case of critical diffusion $s=1/2$ by L. Silvestre in \cite{SilvestreCritical}. Gradient regularity for viscosity solutions of coercive fractional Hamilton-Jacobi equations has been widely analyzed using viscosity solutions' techniques. In \cite{Barles1} the authors have analyzed Lipschitz regularity of solutions via the Ishii-Lions method when $f$ is bounded (which requires the restriction $\gamma<2s$, as for the classical viscous case $s=1$) and via a weak version of the Bernstein method in the periodic setting \cite{BLT}, where $f\in W^{1,\infty}$ in the space variable and $\gamma>1$, even for more general integro-differential operators than fractional powers of the Laplacian. We finally mention that fractional Hamilton-Jacobi-type PDEs and regularity issues have been recently investigated in the framework of periodic homogenization problems \cite{BCT}.  \\
As for the stationary counterpart of \eqref{hjb} with unbounded terms in Lebesgue scales, we mention \cite{AP,Ochoa}.
 Related results for Hamilton-Jacobi equations, even degenerate, can be found in \cite{CG4,Lions85,BardiPerthame,baSW} and the references therein. Other regularity estimates for space-fractional Fokker-Planck equations, also combined with Hamilton-Jacobi equations in the context of Mean Field Games, can be found in \cite{CG1,CCDNV}, while for advection equations with fractional diffusion we refer the reader to \cite{SilvestreAdvectionHolder1,SilvestreAdvectionHolder2}.

\par\smallskip
\textit{Outline}. Section \ref{sec;main} presents a list of the main results and the assumptions used throughout the paper. Section \ref{sec;spaces} is devoted to introduce the main functional spaces and related embedding properties. Section \ref{sec;ffp} concerns the analysis of the well-posedness, Bessel regularity and integrability estimates for fractional Fokker-Planck equations, while Section \ref{sec;fhjb} comprises the applications to regularity issues for equations of Hamilton-Jacobi-type with nonlocal diffusion. Appendix \ref{eqdrift} collects some properties for advection equations with fractional diffusion.
\bigskip

{\bf Acknowledgements.} The author is member of the Gruppo Nazionale per l'Analisi Matematica, la Probabilit\`a e le loro Applicazioni (GNAMPA) of the Istituto Nazionale di Alta Matematica (INdAM). This work has been partially supported by 
the Fondazione CaRiPaRo
Project ``Nonlinear Partial Differential Equations:
Asymptotic Problems and Mean-Field Games''. \\
The author wishes to thank Dr. Marco Cirant for fruitful discussions all along the work, Prof. Luc Molinet and Prof. Baoxiang Wang for discussions about fractional Gagliardo-Nirenberg inequalities in Besov classes.

\section{Assumptions and main results}\label{sec;main}
Throughout all the manuscript we will assume $s\in(1/2,1)$ unless otherwise stated. Our first main results concern the fractional Fokker-Planck equation \eqref{fp}: in the first one, $b$ is assumed to belong to mixed Lebesgue classes in the fractional Aronson-Serrin zone, while in the second result parabolic Bessel regularity is studied in terms of the crossed quantity \eqref{cross}. More precisely, in this section for $\mu\in\R$ we deal with anisotropic spaces of the form
\[
\H_p^{\mu}(Q_T):=\{u\in L^p(0,T;H_p^\mu(\T));\partial_t u\in L^p(0,T;H_p^{\mu-2s}(\T))\}\ ,
\]
where $H_p^\mu(\T)$ is the space of Bessel potentials on the torus. We refer to Section \ref{parspaces} for additional properties of these spaces.\\
We will assume the following additional assumption, referring to Appendix \ref{eqdrift} for further discussions on its validity.
\begin{itemize}
\item[(I)] There exists a unique weak solution to the dual problem of \eqref{fp}
\begin{equation}\label{drift}
\begin{cases}
\partial_t v(x,t)+(-\Delta)^sv(x,t)-b(x,t)\cdot Dv(x,t)=f(x,t)&\text{ in }Q_{\omega,\tau}:=\T\times(\omega,\tau)\ ,\\
v(x,\omega)=v_\omega(x)&\text{ in }\T\ .
\end{cases}
\end{equation}
with $\omega\in[0,\tau)$, where $b$ satisfies \eqref{fractionalAS}. In addition, if $v(\omega)\geq0$, then $v\geq0$ a.e. on $Q_\omega$.
\end{itemize}
\begin{thm}\label{estadjfrac}
Let $b\in L^\sQ(0,\tau; L^\sP(\T))$ with $\sP \in(d/(2s-1),\infty)\text{ and }\sQ\in(2s/(2s-1),\infty]$ satisfying 
\begin{equation*}
\frac{d}{2s\sP}+\frac{1}{\sQ}< \frac{2s-1}{2s}\ ,
\end{equation*}
and $\rho_\tau\in H^{s-1}(\T)$ with $\rho_\tau\in L^1(\T)$. Then, there exists a weak solution $\rho\in\H_2^{2s-1}(Q_\tau)$ to \eqref{fp}. If, in addition, $\rho_\tau\in L^p(\T)$, $p\in(1,\infty]$, then $\rho\in L^\infty(0,\tau;L^p(\T))$. Finally, if (I) holds and $\rho_\tau\geq0$, the solution is unique and $\rho\geq0$ a.e. on $Q_\tau$.
\end{thm}
\begin{thm}\label{estFP2frac}
Let $\rho$ be a (non-negative) weak solution to \eqref{fp} and $1<\sigma'<\frac{d+2s}{2s-1}$.
\begin{itemize}
\item[(i)] There exists $C>0$, depending on $T,\sigma',d,s$ such that
\begin{equation*}
\|\rho\|_{\mathcal{H}_{\sigma'}^{2s-1}(Q_\tau)}\leq C\left(\iint_{Q_\tau} |b(x,t)|^{m'} \rho(x,t) \, dxdt  + \|\rho_\tau\|_{L^1(\T)}\right),
\end{equation*}
where
\[1<\sigma'<\frac{d+2s}{d+2s-1}\]
and 
\begin{equation*}\label{rdq1frac}
m' = 1 + \frac{d+2s}{\sigma(2s-1)}.
\end{equation*}
\item [(ii)]

%Let \[q'\geq\frac{d+2s}{d+2s-1}\] and $\rho_\tau\in L^{p'}(\T)$, $p>1$ arbitrarily large. 

There exists $C>0$, depending on $T,\sigma',d,s$ such that
\begin{equation*}
\|\rho\|_{\mathcal{H}_{\sigma'}^{2s-1}(Q_\tau)}\leq C\left(\iint_{Q_\tau} |b(x,t)|^{m'} \rho(x,t) \, dxdt  +  \|\rho_\tau\|_{L^{p'}(\T)}\right),
\end{equation*}
where either
\[
\sigma'=\frac{d+2s}{d+2s-1}\text{ with any finite }p>1\ ,
\]
or
\[
\sigma'>\frac{d+2s}{d+2s-1}\text{ with }p=\frac{d\sigma}{d+2s-\sigma}
\]
and
\begin{equation*}\label{rdq1frac}
m' = 1 + \frac{d+2s}{\sigma(2s-1)}.
\end{equation*}
\end{itemize}
\end{thm}
As for \eqref{hjb}, we suppose that $H(x,p)$ is $C^1(\T\times\R^d)$, convex in $p$ and has polynomial growth in the gradient entry, i.e.
\begin{equation}\label{H}\tag{H}
\begin{aligned}
& \text{there exist constants $\gamma > 1$ and $C_H>0$ such that}\\
& \qquad C_H^{-1}|p|^{\gamma}-C_H\leq H(x,p)\leq C_H(|p|^{\gamma}+1)\ ,\\
& \qquad D_pH(x,p)\cdot p-H(x,p)\geq C^{-1}_H|p|^{\gamma}-C_H\ , \\
%& \qquad |D_{x}H(x,p)|\leq C_H(|p|^{\gamma}+1) \ , \\
& \qquad C^{-1}_H|p|^{\gamma-1}- {C}_H \le |D_{p}H(x,p)|\leq C_H|p|^{\gamma-1}+{C}_H \ ,
\end{aligned}
\end{equation}
for every $x\in \T$, $p\in\R^d$. Moreover, we suppose without loss of generality that $H \ge 0$. Recall that the Lagrangian $L : \T \times \R^d \to \R$, $L(x, \nu) := \sup_{p} \{p \cdot \nu - H(x, p)\}$, namely the Legendre transform of $H$ in the $p$-variable, is well defined by the superlinear behavior of $H(x, \cdot)$. Moreover, by convexity of $H(x, \cdot)$,
\[
H(x, p) = \sup_{\nu \in \R^d} \{\nu \cdot p - L(x, \nu)\},
\]
and 
\begin{equation}\label{HL}
H(x, p) = \nu \cdot p - L(x, \nu) \quad \text{if and only if} \quad \nu = D_p H(x, p).
\end{equation}
The following properties of $L$ are standard (see, e.g. \cite{CS}): for some $C_L > 0$,
\begin{align}
\tag{L1}\label{L1} & C_L^{-1}|\nu|^{\gamma'} - C_L \le L(x, \nu) \le C_L |\nu|^{\gamma'}\\
%\tag{L2}\label{L2} & |D_x L(x, \nu)| \le C_L (|\nu|^{\gamma'} + 1).
\end{align}
for all $\nu \in \R^d$.\\
 Concerning the case $\gamma \ge 2s$ and when dealing with H\"older regularity, we will impose some additional space regularity, i.e for $\alpha \in (0,1)$ to be determined,
\begin{equation}\label{Ha}\tag{$H_\alpha$}
H(x, p) - H(x+\xi, p)  \le C_H |\xi|^\alpha \big( |D_pH(x,p)|^{\gamma'}+1 \big)
\end{equation}
for all $x,\xi \in \T$ and $p \in \R^d$. A prototype example of $H$ satisfying \eqref{H} is
\[
H(x,p) = h(x)|p|^{\gamma} + b(x) \cdot p, \qquad 0 < h_0 \le h(x), \quad h, b \in C(\T).
\]
Note that whenever $h \in C^\alpha(\T)$, this model Hamiltonian satisfies also \eqref{Ha}.
Unless otherwise stated, in the next results we will always assume that (I) holds to have the full well-posedness of the adjoint problem. However, our approach via duality makes use only of the existence of positive solutions to \eqref{fp}.
Then, our first results for \eqref{hjb} concern sup-norm and integral estimate for strong solutions as in Definition \ref{strong} when $f\in L^q$, obtained via the nonlinear adjoint method via the strategy already implemented in \cite{CG2,CG5}.
\begin{thm}\label{supnorm}
Let \eqref{H} be in force and $\gamma>1$, $q>\frac{d+2s}{2s}$. Then, there exists a constant $C>0$ depending on $T,d,s,q,\|f\|_{L^q(Q_T)}, \|u_0\|_{C(\T)}$ such that any global weak solution to \eqref{hjb} satisfies
\[
\|u(\cdot,\tau)\|_{C(\T)}\leq C\text{ for all }\tau\in[0,\tau]\ .
\]
The estimate holds even for strong solutions to \eqref{hjb}.
\end{thm}
\begin{thm}\label{integest}
Let \eqref{H} be in force and $\gamma\in(1,2s)$. Let also $u\in\H_q^{2s}(Q_T)$ be a strong solution to \eqref{hjb} with $q>\frac{d+2s}{(2s-1)\gamma'}$. The following assertions hold:
\begin{itemize}
\item[(A)] There exists a constant $C_1>0$, depending on $T,d,s,C_H, \|f\|_{L^q(Q_\tau)},\|u_0\|_{L^p(\T)}$, $q=\frac{d+2s}{2s}$, such that any strong solution to \eqref{hjb} satisfies
\[
\|u\|_{L^\infty(0,\tau;L^p(\T))}\leq C_1
\]
for $1\leq p<\infty$.
\item[(B)] There exists a constant $C_2>0$, depending on $T,d,s,C_H \|f\|_{L^q (Q_\tau)}, \|u_0\|_{L^{p}(\T)}$, such that any strong solution to \eqref{hjb} satisfies
\[
\|u\|_{L^\infty(0,\tau;L^{p}(\T))}\leq C_2
\]
with $p=\frac{dq}{d+2s-2sq}$ if $q<\frac{d+2s}{2s}$.
\end{itemize}

\end{thm}
Owing to a similar approach and following the scheme of \cite{CG5}, spatial H\"older's regularity estimates for weak energy solutions to fractional Hamilton-Jacobi equations with unbounded right-hand side are provided. 
\begin{thm}\label{mainholder}
Assume \eqref{H}, $(H_\alpha)$ and $\gamma\geq2s$.

\begin{itemize} 
\item[(C)]  Let $f\in L^q(Q_T)$ with $q>\frac{d+2s}{(2s-1)\gamma'}$. Let $u$ be a local weak solution to \eqref{hjb}. Then, there exists $C_1>0$ depending on  $t_1,C_H$, $\|u\|_{C(\overline{Q}_T)}$, $\|f\|_{L^q(Q_T)}$, $q,d,T,s$ such that
\[
\mathrm{sup}_{t\in(t_1,T)}[u(\cdot,t)]_{C^\alpha(\T)}\leq C_1
\]
with $\alpha,q,\gamma$ as follows:\\
\begin{itemize}
\item For $\gamma=2s$ we get $\alpha= 2s-\frac{d+2s}{q}$ such that $\alpha\in(0,2s-1)$ when $q<d+2s$, $\alpha<1$ when $d+2s\leq q<\frac{d+2s}{2s-1}$, while we have any $\alpha\in(0,1)$ when $q\geq\frac{d+2s}{2s-1}$;
\item For $2s\leq \gamma<\frac{1}{2-2s}$ (i.e. $\gamma'(2s-1)>1$), we get $\alpha= \gamma'(2s-1)-\frac{d+2s}{q}$ such that $\alpha\in(0,2s-1)$ if $q<\frac{d+2s}{(2s-1)(\gamma'-1)}$, $\alpha\in(2s-1,1)$ if $\frac{d+2s}{(2s-1)(\gamma'-1)}\leq q<\frac{(d+2s)(\gamma-1)}{1-(2-2s)\gamma}$, while we have any $\alpha\in (0,1)$ if $q\geq \frac{(d+2s)(\gamma-1)}{1-(2-2s)\gamma}$;
\item For $\gamma\geq\frac{1}{2-2s}$ (i.e. $\gamma'(2s-1)\leq1$), we get $\alpha=\gamma'(2s-1)-\frac{d+2s}{q}\in(0,1)$ whenever $q>\frac{d+2s}{\gamma'(2s-1)}$.
\end{itemize}
\item[(D)] If $u$ is a strong solution in $\H_q^{2s}(Q_T)$ with $q>\frac{d+2s}{(2s-1)\gamma'}$ to \eqref{hjb}, then $u\in L^\infty(0,T;C^\alpha(\T))$. In particular, there exists a positive constant $C_2$ depending on $C_H$, $\|u_0\|_{C^\alpha(\T)}$, $\|f\|_{L^q(Q_T)}$, $q,d,T,s$ such that
\[
\mathrm{sup}_{t\in(0,T)}[u(\cdot,t)]_{C^\alpha(\T)}\leq C_2\ .
\]
with $\alpha,\gamma$ as in (C).
\end{itemize}
\end{thm}
We remark that
\[
\alpha= \gamma'(2s-1)-\frac{d+2s}{q}\to \alpha=\gamma'-\frac{d+2}{q}\text{ and } q=\frac{(d+2s)(\gamma-1)}{1-(2-2s)\gamma}\to  \frac{d+2}{\gamma'-1}\text{ as }s\to 1^-\ ,
\]
so that we find the same thresholds in \cite{CG5}.\\
In the last part of the paper, we address optimal $L^q$-regularity in the form of a priori estimates for strong solutions, that is stated in the next
\begin{thm}\label{maxregmain}
Assume \eqref{H}, $(H_\alpha)$. Let $u\in\H_q^{2s}(Q_T)$ be a strong solution to \eqref{hjb} with $q>\frac{d+2s}{(2s-1)\gamma'}$, and assume that there exists $\tilde K_1>0$ such that
\[
\|f\|_{L^q(Q_T)}+\|u_0\|_{W^{2s-\frac{2s}{q},q}(\T)}\leq \tilde K_1\ .
\]
If
\[
q>\begin{cases}
\frac{(d+2s)(\gamma-1)}{(2s-1)\gamma}&\text{ if }\gamma\in(\frac{d+6s-2}{d+2s},2s)\\
\frac{d+2s}{2s}&\text{ if }\gamma=2s\\
\frac{(d+2s)(\gamma-1)}{(2s-2)\gamma+2s}&\text{ if }\gamma\in\left(2s,\frac{s}{1-s}\right)\ ,
\end{cases}
\]
then there exists a constant $C_1>0$ depending on $\tilde K_1,q,d,T,s,C_H$ such that
\[
\|\partial_tu\|_{L^q(Q_T)}+\|u\|_{L^q(0,T;H_q^{2s}(\T))}+\||Du|^\gamma\|_{L^{q}(Q_T))}\leq C_1\ .
\]

\end{thm}
Note that
\[
\frac{(d+2s)(\gamma-1)}{(2s-2)\gamma+2s}\to \frac{(d+2)(\gamma-1)}{2}\text{ and }\frac{1}{1-s}\to \infty\text{ as }s\to 1^-
\]
so that we recover the same thresholds in \cite{CG5}. \\
We finally point out that our results, and in particular the H\"older bounds, apply to the so-called fractional Kardar-Parisi-Zhang equations, see \cite{KW,PPS}, of the form
\[
\partial_t z+(-\Delta)^s z=G(x,Dz(x,t))-f(x,t)\text{ in }Q_T
\]
where $G$ satisfies \eqref{H}. In other terms, the sign in front of $H$ is not important since $u=-z$ solves \eqref{hjb} with $H(x,p)=G(x,-p)$.

\section{Functional spaces}\label{sec;spaces}
\subsection{Stationary spaces: definitions and useful results}
We denote by $L^p(\T)$ the space of all measurable and periodic functions belonging to $L^p_{loc}(\R^d)$ endowed with the norm $\|\cdot\|_p=\|\cdot\|_{L^p((0,1)^d)}$. Let $k$ be a nonnegative integer. We denote by $W^{k,p}(\T)$ the space of $L^p(\T)$ functions with distributional derivatives in $L^p(\T)$ up to order $k$. For $\mu\in\R$ and $p\in(1,\infty)$, the space of Bessel potentials $H_p^{\mu}(\T)$ comprises those distributions verifying the integrability condition $(I-\Delta)^{\frac{\mu}{2}}u\in L^p(\T)$, where $(I-\Delta)^{\frac{\mu}{2}}$ is the operator defined in terms of Fourier series as
\[
(I-\Delta)^{\frac{\mu}{2}}u=\sum_{k\in\Z^d}(1+4\pi^2k^2)^\frac{\mu}{2}\hat u(k)e^{2\pi i k\cdot x}\ ,
\]
with
\[
\hat u(k)=\int_\T u(x)e^{-2\pi i k\cdot x}\ .
\]
We denote the norm in $H_p^{\mu}(\T)$ as
\[
\|u\|_{\mu,p}:=\|(I-\Delta)^{\frac{\mu}{2}}u\|_p\simeq \|u\|_p+\|(-\Delta)^{\frac{\mu}{2}}u\|_p\ .
\]
The proof of the latter equivalence is given in \cite[Remark 2.3]{CG1}. Let us also remark that when $\mu=k$ is a nonnegative integer, $W^{k,p}$ is isomorphic to $H_p^k$, see e.g. \cite[Remark 2.3]{CG1}. Moreover, it can be seen that the operator $(I-\Delta)^{\frac{\mu}{2}}$ maps isometrically $H_p^{\eta+\mu}$ in $H_p^\eta$ for any $\eta,\mu\in\R$, see again  \cite[Remark 2.3]{CG1} for the proof. We further recall that another characterization of spaces of Bessel potentials can be given via complex interpolation methods, namely
\[
H_p^\mu(\T)\simeq [L^p(\T),W^{k,p}(\T)]_{\theta}\ ,\mu=k\theta\ ,
\]
where $[X,Y]_\theta$ stands for the complex interpolation space among the Banach spaces $(X,Y)$, see e.g. \cite{Lunardi,BL} for a complete account.\\
Let now $\mu\in(0,1)$ and $1\leq p,q\leq \infty$. The Besov space $B_{pq}^\mu(\T)$ consists of all functions $u\in L^p(\T)$ such that
the norm
\[
\|u\|_{B_{pq}^\mu(\T)}:=\|u\|_{L^p(\T)}+\left(\int_{\T}\frac{\|f(x+h)-f(x)\|^q_{L^p(\T)}}{|h|^{d+\mu q}}\,dh\right)^{\frac{1}{q}}
\]
is finite. When $p=q=\infty$ and $\mu\in(0,1)$ we have $B_{\infty\infty}^\mu(\T)\simeq C^\mu(\T)$ (cf \cite[Section 3.5.4 p. 168-169]{ST}), i.e. the classical H\"older space, which is endowed with the equivalent norm
\[
\|u\|_{C^\alpha(\T)}:=\|u\|_{\infty}+\sup_{x\neq y\in\T}\frac{|u(x)-u(y)|}{\mathrm{dist}(x,y)^\alpha}\ ,
\]
where $\mathrm{dist}(x,y)$ is the geodesic distance among $x,y\in\T$. When $p=q$ and $\mu$ is not an integer, one has $B_{pp}^\mu(\T)\simeq W^{\mu,p}(\T)$, where $W^{\mu,p}(\T)$ is the classical Sobolev-Slobodeckii scale in the periodic setting, see \cite[p.13]{Lunardi}. When $q=\infty$ and $p$ is finite, the space $B_{p\infty}^\mu(\T)\simeq N^{\mu,p}(\T)$ is known as Nikol'skii space \cite{Nik} and the aforementioned norm is interpreted in the usual sense via
\[
\|u\|_{N^{\mu,p}(\T)}:=\|u\|_{L^p(\T)}+\sup_{|h|>0}|h|^{-\mu}\|u(x+h)-u(x)\|_{L^p(\T)}\ ,
\]
see \cite[Chapter 17]{Leoni} for the whole space case and \cite[p. 460]{Taibleson2}, \cite[Section 3.5.4]{ST} for the definition in the periodic case.
Yet another characterization of Besov classes can be given by means of real interpolation methods. For $m\in\N$, $p,q\in[1,\infty]$ and $\theta\in(0,1)$ we have
\[
(L^p(\T),W^{m,p}(\T))_{\theta,q}\simeq B^{\theta m}_{pq}(\T)\ ,
\]
where $(X,Y)_{\theta,q}$ stands for the real interpolation space of the interpolation couple of Banach spaces $(X,Y)$, with equivalence of the respective norms, see e.g. \cite[Example 1.10]{Lunardi}, \cite[Theorem 17.24]{Leoni}. We also denote by $F^\mu_{pq}(\T)$ the periodic Triebel-Lizorkin scale, and refer to \cite[Section 3.5.2]{ST} for its definition.

We recall some standard embeddings we will use in the sequel among the aforementioned spaces.
\begin{lemma}\label{inclstat}
\begin{itemize}
  \item[(i)] Let $\nu,\mu\in\R$ with $\nu\leq\mu$, then $H_p^{\mu}(\T)\hookrightarrow H_p^{\nu}(\T)$.
  \item[(ii)] If $p\mu>d$ and $\mu-d/p$ is not an integer, then $H_p^{\mu}(\T)\hookrightarrow C^{\mu-d/p}(\T)$.
  \item[(iii)] Let $\nu,\mu\in\R$ with $\nu\leq\mu$, $p,q\in(1,\infty)$ and
  \[
  \mu-\frac{d}{p}=\nu-\frac{d}{q}\ ,
  \]
  then $H_p^{\mu}(\T)\hookrightarrow H_q^{\nu}(\T)$.
   \item[(iv)] If $p\mu=d$,  then $H_p^{\mu}(\T)\hookrightarrow L^q(\T)$ for all $1\leq q<\infty$.
\end{itemize}
\end{lemma}
\begin{proof}
For (i)-(iii) see \cite[Lemma 2.5]{CG1}, \cite{ST} and the references therein, while for (v) see \cite{AH}.
\end{proof}

\begin{lemma}\label{inclstatW}
\begin{itemize}
  \item[(i)] Let $\nu,\mu\in\R$ with $\nu\leq\mu$, then $W^{\mu,p}(\T)\subset W^{\nu,p}(\T)$.
  \item[(ii)] If $p\mu>d$ and $\mu-d/p$ is not an integer, then $W^{\mu,p}(\T)\subset C^{\mu-d/p}(\T)$.
  \item[(iii)] Let $\nu,\mu\in\R$ with $\nu\leq\mu$, $p,q\in(1,\infty)$ and
  \[
  \mu-\frac{d}{p}=\nu-\frac{d}{q}\ ,
  \]
  then $W^{\mu,p}(\T)\subset W^{\nu,q}(\T)$.
  \item[(iv)] If $p\mu=d$, then $W^{\mu,p}(\T)\subset L^q(\T)$ for all $1\leq q<\infty$.
\end{itemize}
\end{lemma}
\begin{proof}
For (i)-(iii) we refer to \cite[Section 3.5.5]{ST} and \cite{ST}, while for (iv) see \cite[Theorem 6.9]{DNPV}.
\end{proof}

\begin{lemma}\label{inclBesovBessel}
We have the following inclusions for $\mu\in\R$.
\begin{itemize}
  \item[(i)] $B^{\mu}_{pp}(\T)\subseteq H_p^{\mu}(\T)\subseteq B_{p,2}^\mu(\T)$ for $1<p\leq 2$. 
  \item[(ii)] $B_{p,2}^\mu(\T) \subseteq H_p^{\mu}(\T)\subseteq B^{\mu}_{pp}(\T)$ for $2\leq p<\infty$.
  \item[(iii)] $B_{p1}^\mu(\T)\subseteq H_p^\mu(\T)\subseteq B^\mu_{p\infty}(\T)$ for $1\leq p\leq\infty$.
\end{itemize}
\end{lemma}
\begin{proof}
The result on $\R^d$ is proven in \cite[Section 2.3.3]{trbookinterpolation} (see also \cite[Theorem 6.4.4]{BL}. Recalling that $H_p^{\mu}$ is isomorphic to a Triebel-Lizorkin scale (see \cite[Theorem 3.5.4-(v)]{ST} and the same chapter for the definition of this space), one uses \cite[Remark 3.5.1.4-(20)]{ST} to show (i) and (ii). Property (iii) is proved in \cite[Theorem 6.2.4]{BL} and \cite[Theorem 20]{Taibleson2}.
\end{proof}
We now recall the following compactness result.
\begin{lemma}
For $\mu>0$ such that $\mu p<d$ and $1<q<\frac{dp}{d-\mu p}$ the embedding of $H_p^\mu(\T)$ onto $L^q(\T)$ is compact.
\end{lemma}
\begin{proof}
When $\mu=k\in\N$ this is the classical Rellich-Kondrachov theorem \cite{Leoni}. We restrict to consider $\mu\in(0,1)$. When $p=2$ the result can be deduced by the fact that $H_2^\mu(\T)\simeq W^{\mu,2}(\T)$, and by classical compactness properties of real interpolation spaces, cf \cite[Section 1.16.4]{trbookinterpolation} applied with $A=(L^2(\T),W^{1,2}(\T))_{\theta,2}$, $A_0=B=L^2(\T)$, $A_1=W^{1,2}(\T)$, using the compactness of $W^{1,2}(\T)$ onto $L^2(\T)$ that give the compact embedding of $H_2^\mu$ onto $L^2$. Then, the compact embedding of $H_2^\mu$ onto $L^q$, $1<q<2d/(d-2\mu)$ follows by interpolation as we will show in the case $p\neq 2$ below. The general case can be handled as follows. 
It is well-known that $W^{1,p}(\T)$ is compactly embedded onto $L^r(\T)$ for all $r$ such that $1<r<\frac{dp}{d-p}$ by Rellich-Kondrachov Theorem, and hence the identity map $T:W^{1,p}(\T)\to L^p(\T)$, $T(u)=u$ is compact. Moreover, $T$ is also continuous from $L^p(\T)$ onto itself. 
Thus, one first recalls that Besov spaces can be defined via real interpolation as follows $B^{\sigma}_{pq}(\T)=(L^p(\T),W^{1,p}(\T))_{\sigma,q}$. Then, one may compose continuous embeddings from Lemma \ref{inclBesovBessel}-(i) and -(ii) with compact embeddings for fractional Sobolev spaces $W^{\mu,p}(\T)$ onto $L^p(\T)$ as obtained in \cite{AmannCompact}. The latter can be deduced in turn via compactness results for real interpolation spaces \cite[Section 1.16.4]{trbookinterpolation} as in the case $p=2$. Therefore, we have the compact embedding of $H_p^\mu(\T)$ onto $L^p(\T)$.
We now take a bounded sequence $u_n$ in $H_p^\mu(\T)$. Therefore, one can extract a subsequence $u_{n_k}$ converging strongly in $L^p(\T)$. By interpolation, for every $p<q<\frac{dp}{d-\mu p}$, there exists $\theta\in(0,1)$ such that
\[
\|u_{n_k}-u_{n_j}\|_{q}\leq \|u_{n_k}-u_{n_j}\|_{p}^{1-\theta}\|u_{n_k}-u_{n_j}\|_{\frac{dp}{d-\mu p}}^\theta\to0
\]
as $j,k\to\infty$ since $u_{n_k}$ is bounded in $H_p^\mu(\T)$, which is in turn continuously embedded onto $L^{\frac{dp}{d-\mu p}}(\T)$ by Lemma \ref{inclstat}-(iii). Then we have the strong convergence in $L^q$ with $q$ as above, as desired. \end{proof}
We first recall the following Gagliardo-Nirenberg inequality
\begin{lemma}\label{GN1}
Let $1<q,r<\infty$, $1< z\leq\infty$ and $s\in(1/2,1)$. Let $u\in H_q^{2s}(\T)\cap L^z(\T)$. Then, there exists a constant $C$ depending on $d,q,z,s,r$ and $\theta\in(0,1)$ such that
\begin{equation}\label{in1}
\|u\|_{W^{1,r}(\T)}\leq C\|u\|_{H^{2s}_{q}(\T)}^\theta\|u\|_{L^z(\T)}^{1-\theta}\ .
\end{equation}
where
\[
\frac1r=\frac{1}{d}+\theta\left(\frac1q-\frac{2s}{d}\right)+\frac{1-\theta}{z}\ , 1\leq 2s\theta\ .
\]
\end{lemma}
\begin{proof}

Inequality \eqref{in1} has been obtained in \cite[Corollary 1.5]{HMOW} and \cite[Theorem 6]{Leonori} (for $j=0$ and $\mu=2s$) on the whole space for $r,z\in(1,\infty)$, the periodic case being obtained via extension arguments, see e.g. \cite[Lemma A.3]{CG5} or \cite[Lemma 2.5]{CG1}. When e.g. $q\geq 2$, the endpoint case $z=\infty$ can be deduced from the results in \cite[Corollary 3.2-(c)]{BM} owing to the inclusion $H^{2s}_q\hookrightarrow W^{2s,q}$.
\end{proof}

We now provide a Gagliardo-Nirenberg interpolation inequality involving H\"older and Bessel potential scales. 
\begin{lemma}\label{GN2}
Let $u\in H^{2s}_q(\T)\cap C^{\alpha}(\T)$, $q\in(1,\infty)$, $\alpha\in(0,1)$. There exists a constant $c>0$ depending on $d,\beta,\alpha,q,s,p$ such that
\[
\|u\|_{W^{1,p}(\T)}\leq c\|u\|_{H^{2s}_q(\T)}^\beta\|u\|_{C^{\alpha}(\T)}^{1-\beta}\ ,
\]
when the following compatibility conditions hold
\[
\frac1p=\frac{1}{d}+(1-\beta)\left(\frac1q-\frac{2s}{d}\right)-\beta\frac{\alpha}{d}\ ,
\]
with 
\[
\beta\in\left[\frac{1-\alpha}{2s-\alpha},1\right)\ ,\alpha\neq 2s-\frac{d}{q}
\]
\end{lemma}
\begin{proof}
We first prove the inequality on $\R^d$, the periodic case being again a consequence of an extension argument. We use \cite[Theorem 4.1]{HMOW} saying that
\[
\|u\|_{B^{\nu}_{pq}(\R^d)}\leq C\|u\|_{B^{\nu_0}_{p_0\infty}(\R^d)}^{1-\beta}\|u\|_{B^{\nu_1}_{p_1\infty}(\R^d)}^\beta
\]
holds if
\[
\frac{d}{p}-\nu=(1-\beta)\left(\frac{d}{p_0}-\nu_0\right)+\beta \left(\frac{d}{p_1}-\nu_1\right);
\]
\[
\nu_0-\frac{d}{p_0}\neq \nu_1-\frac{d}{p_1};
\]
\[
\nu\leq (1-\beta)\nu_0+\beta \nu_1\ .
\]
We take $p_0=\infty$, $q=1$, $\nu_1=2s$, $\nu=1$, $p_1=q$ to get
\[
\|u\|_{B^{1}_{p1}(\R^d)}\leq C_1\|u\|_{B^{\alpha}_{\infty\infty}(\R^d)}^{1-\beta}\|u\|_{B^{2s}_{q\infty}(\R^d)}^\beta
\]
On one hand, we use the embedding $H_{q}^{2s}\hookrightarrow B^{2s}_{q\infty}$ from Lemma \ref{inclBesovBessel}-(iii) and conclude
\[
\|u\|_{B^{1}_{p1}(\R^d)}\leq C_2\|u\|_{B^{\alpha}_{\infty\infty}(\R^d)}^{1-\beta}\|u\|_{H^{2s}_{q}(\R^d)}^\beta\ .
\]
On the other hand, owing to the embedding $B^{1}_{p1}\hookrightarrow H_{p}^{1}$, together with the fact that $B^{\alpha}_{\infty\infty}\simeq C^\alpha$, we conclude
\[
\|u\|_{H^1_p(\R^d)}\leq C_3\|u\|_{C^{\alpha}(\R^d)}^{1-\beta}\|u\|_{H^{2s}_{q}(\R^d)}^\beta
\]
for
\[
\frac{1}{p}=\frac{1}{d}-(1-\beta)\frac{\alpha}{d}+\beta \left(\frac{1}{q}-\frac{2s}{d}\right)
\]
with $\beta\in(0,1)$,
\begin{equation}\label{betaalpha}
1\leq (1-\beta)\alpha+\beta 2s
\end{equation}
and
\[
\alpha\neq 2s-\frac{d}{q}\ .
\]
From \eqref{betaalpha} we deduce $\beta\geq \frac{1-\alpha}{2s-\alpha}$.\end{proof}
\begin{rem}
The above inequality has been obtained in \cite{Miranda,NHolder} when $s=1$, see also \cite{Maugeri}.
\end{rem}

We conclude this section with a fractional Poincar\'e-Wirtinger inequality.
\begin{lemma}\label{PW}
Let $U\subset\R^d$ be a cube. There exists $C=C(d,p,\mu,U)$ such that for $\mu\in(0,1]$ and $p\in(1,\infty)$
\[
\|u-u_U\|_{L^p(U)}\leq C_1[u]_{W^{\mu,p}(U)}\ ,
\]
where $u_U=\fint_{U} u\,dx$ and $[\cdot]_{W^{\mu,p}(U)}$ stands for the Gagliardo seminorm.  As a consequence, when $U=\T$, for $s\in(1/2,1)$ there exists $C_2>0$ such that 
\[
\|u-u_U\|_{L^p(\T)}\leq C_2\|(-\Delta)^{s-\frac12}u\|_{L^p(\T)}\ .
\]
\end{lemma}
\begin{proof}
The first inequality can be found in \cite[Proposition B.11]{Armstrong}, \cite[Chapter 17]{Leoni}. The second one follows from the first and the inclusions among Bessel and Sobolev-Slobodeckii spaces $H_p^{\mu+\epsilon}(\T)\subseteq W^{\mu,p}(\T)\subseteq H_p^{\mu-\epsilon}(\T)$, $\epsilon>0$, $\mu\in\R$, cf \cite[Lemma 2.14]{CG1}.
\end{proof}
\subsection{Parabolic spaces: definitions and embeddings}\label{parspaces}
In this section we introduce some functional spaces involving time and space weak derivatives. Let again $\mu\in\R$ and $p \in (1, \infty)$. We denote by $\mathbb{H}_{p}^{\mu}(Q):=L^p(0,T;H_p^{\mu}(\T))$, $Q:=\T\times I$, the space of measurable functions $u:(0,T)\rightarrow H_p^{\mu}(\T)$ endowed with the norm
\begin{equation*}
\norm{u}_{\mathbb{H}_p^{\mu}(Q)}:=\left(\int_0^T\norm{u(\cdot,t)}^p_{H^{\mu}_p(\T)}dt\right)^{\frac1p}\ .
\end{equation*}
We define the space $\H_p^{\mu}(Q)$ as the space of functions $u\in \mathbb{H}_p^{\mu}(Q)$ with $\partial_tu\in(\mathbb{H}_{p'}^{2s-\mu}(Q))'$ equipped with the norm
\begin{equation*}
\norm{u}_{\mathcal{H}_p^{\mu}(Q)}:=\norm{u}_{\mathbb{H}_p^{\mu}(Q)}+\norm{\partial_tu}_{(\mathbb{H}_{p'}^{2s-\mu}(Q))'}\ .
\end{equation*}
We refer the reader to \cite{CG1,CL}. These are natural spaces in the standard parabolic setting $s=1$: when $s=1$ and $\mu=2$ we have $\H_p^2\simeq W^{2,1}_p$, cf \cite{LSU}, see also \cite{k1,k2}, \cite{CT}, \cite[Chapter 6]{BKRS} for further properties in the case $s = 1$. Note that $(\mathbb{H}_{p'}^{2s-\mu}(Q))'$ coincides with $\mathbb{H}_{p}^{\mu-2s}(Q)$. In the sequel we denote by $Q_\tau=\T\times(0,\tau)$ and $Q_{\omega,\tau}=\T\times(\omega,\tau)$.\\
We now recall the following trace result of functions in $\H_p^\mu$ on the hyperplane $t=0$, which extends  \cite[Lemma II.3.4]{LSU} for classical spaces associated to heat PDEs to the fractional framework.
\begin{lemma}\label{trace}
If $u\in\H_p^{\mu}(Q_T)$, $\mu\in\R$ and $p>1$ satisfying $\mu-2s/p>0$, then $u(0)\in W^{\mu-2s/p,p}(\T)$. In addition, the space $\H_p^\mu$ is continuously embedded onto $C([0,T];W^{\mu-\frac{2s}{p}}(\T))$.
\end{lemma} 
\begin{proof}
The first statement is a consequence of \cite[Corollary 1.14]{Lunardi}, the embedding properties for the domain of the fractional Laplacian $D(-(-\Delta)^s)$ and the Reiteration Theorem in interpolation theory. The second fact can be deduced again by \cite[Corollary 1.14]{Lunardi}, see also \cite[Theorem III.4.10.2]{Amann} and \cite{CG5} for the case $s=1$.
\end{proof}
We now recall some fractional parabolic embedding theorems partially proved in \cite{CG1} and in \cite{TesiAle}. We adapt the interpolation approach proposed in \cite{k1,k2}.
\begin{lemma}\label{embs}
\begin{itemize}
\item[(i)] If $1<p<\frac{d+2s}{\mu}$, then $\H_p^{\mu}(Q_T)$ is continuously embedded onto $L^q(Q_T)$ for $1\leq q\leq \frac{(d+2s)p}{d+2s-\mu p}$.
\item[(ii)] For $p\geq 2$ the parabolic space $\H_p^{2s-1}(Q_T)$ is continuously embedded into $L^\delta(0,T;W^{\alpha,\delta}(\T))$, where $\delta>p$ and
\[
\alpha=2s-1+\frac{d+2s}{\delta}-\frac{d+2s}{p}\ ,
\]
while for $p\in(1,2]$ we have the embedding of $\H_p^{2s-1}(Q_T)$ into $L^\delta(0,T;H^{\alpha}_{\delta}(\T))$ with $\delta>p$ and $\alpha$ as above.
\item[(iii)] If $p>\frac{d+2s}{\mu}$, then $\H_p^{2s-1}(Q_T)$ is continuously embedded onto $C^{\alpha,\alpha/2s}$ for some $\alpha\in(0,1)$. Moreover, the space $\H_p^{2s}(Q_T)$ is continuously embedded onto $C([0,T];C^{2s-\frac{d+2s}{p}}(\T))$.
\end{itemize}
\end{lemma}
\begin{proof}
(i) is proven in \cite[Proposition 2.11]{CG1} for $1\leq q<\frac{(d+2s)p}{d+2s-\mu p}$. A slight modification of that proof allows even to prove the endpoint case $q=\frac{(d+2s)p}{d+2s-\mu p}$, which we provide below.
Here, we distinguish the cases $1<p\leq 2$ and $2< p<\infty$ in view of the inclusions stated in Lemma \ref{inclBesovBessel}. 
To prove the first case $1<p\leq 2$, we note that for any $\theta \in (0, 1)$, if $\nu = \nu(\theta)=(\mu-2s/p)(1-\theta)+\mu\theta$, then $H^\nu_p$ can be obtained by complex interpolation between $H^{\mu}_p$ and $H^{\mu- 2s/p}_p$ (see, e.g., \cite[Theorem 6.4.5]{BL}). Moreover, $H^\nu_p$ is continuously embedded in $H^{\nu+d/q-d/p}_q$ in view of Lemma \ref{inclstat}. Hence, for a.e. $t$,
\begin{equation*}
c(d, p, s,q)\norm{u(t)}_{\nu-\frac{d}{p}+\frac{d}{q},q}\leq\norm{u(t)}_{\nu,p}\leq \norm{u(t)}_{\mu- 2s/p,p}^{1-\theta}\norm{u(t)}_{\mu,p}^{\theta}.
\end{equation*}
Therefore, for all $\alpha \le \nu-\frac{d}{p}+\frac{d}{q} = \mu+\frac{d}{q}-\frac{d+2s(1-\theta)}{p}$,
\begin{multline*}
\left(\int_0^T\norm{u(t)}_{\alpha,q}^{\frac{p}{\theta}}dt\right)^{\theta}\leq
C_1\left(\int_0^T\norm{u(t)}_{\mu-2s/p,p}^{(1-\theta)\frac{p}{\theta}}\norm{u(t)}_{\mu,p}^pdt\right)^{\theta} \\
\leq C_2\left(\int_0^T\norm{u(t)}_{W^{\mu-2s/p,p}(\T)}^{(1-\theta)\frac{p}{\theta}}\norm{u(t)}_{\mu,p}^pdt\right)^{\theta}\ ,
\end{multline*}
where we used that for $1<p\leq 2$, $W^{\mu-2s/p,p}$ is embedded onto $H_p^{\mu-2s/p}$ (cf Lemma \ref{inclBesovBessel}-(i)). Then, the last inequality is less than or equal to
\begin{equation*}
C\sup_{t\in[0, T]}\norm{u(t)}_{W^{\mu-2s/p,p}(\T)}^{(1-\theta)p}\left(\int_0^T\norm{u(t)}_{\mu,p}^pdt\right)^{\theta} \leq C\norm{u}_{\mathcal{H}_p^{\mu}(Q_T)}^{(1-\theta)p}\norm{u(t)}_{\mathbb{H}_p^{\mu}(Q_T)}^{\theta p} \leq C\norm{u}_{\mathcal{H}_p^{\mu}(Q_T)}^{p}\ ,
\end{equation*}
where, in the second inequality we used the embedding in Lemma \ref{trace}
\[
\H_p^{\mu}(Q_T)\hookrightarrow C([0,T];W^{\mu-2s/p,p}(\T))
\]
while, in the last one, Young's inequality.\\
As for the case $p\geq 2$, we interpolate in the Sobolev-Slobodeckii scale. In particular, one uses that $W^{\nu,p}$ can be obtained by real interpolation among $W^{\mu,p}$ and $W^{\mu- 2s/p,p}$. Moreover, $W^{\nu,p}$ is continuously embedded in $W^{\nu+d/q-d/p,q}$ in view of Lemma \ref{inclstatW}-(iii). Hence, for a.e. $t$,
\begin{equation*}
c(d, p, s, q)\norm{u(t)}_{W^{\nu-\frac{d}{p}+\frac{d}{q},q}(\T)}\leq\norm{u(t)}_{W^{\nu,p}(\T)}\leq \norm{u(t)}_{W^{\mu- 2s/p,p}(\T)}^{1-\theta}\norm{u(t)}_{W^{\mu,p}(\T)}^{\theta}.
\end{equation*}
Then, for all $\alpha$ verifying $\alpha \le \nu-\frac{d}{p}+\frac{d}{q} \leq \mu+\frac{d}{q}-\frac{d+2s(1-\theta)}{p}$ we have
\begin{multline*}
\left(\int_0^T\norm{u(t)}_{W^{\alpha,q}(\T)}^{\frac{p}{\theta}}dt\right)^{\theta}\leq C_1\left(\int_0^T\norm{u(t)}_{W^{\nu-\frac{d}{p}+\frac{d}{q},q}(\T)}^{\frac{p}{\theta}}dt\right)^{\theta}\\
\leq
C_2\left(\int_0^T\norm{u(t)}_{W^{\mu- 2s/p,p}(\T)}^{(1-\theta)p}\norm{u(t)}_{W^{\mu,p}(\T)}^p dt \right)^{\theta} \\
 \leq C_3\sup_{t\in[0,T]}\norm{u(t)}_{W^{\mu- 2s/p,p}(\T)}^{(1-\theta)p}\left(\int_0^T\norm{u(t)}_{\mu,p}^p dt\right)^{\theta}
\end{multline*}
where we used that $H_p^{\mu}$ is embedded onto $W^{\mu,p}$ when $p\geq 2$ (see Lemma \ref{inclBesovBessel}). At this stage, one has to use the maximal regularity embedding in Lemma \ref{trace} to get
\[
\H_p^{\mu}(Q_T)\hookrightarrow C([0,T];W^{\mu-2s/p,p}(\T))
\]
and finally conclude the assertion setting $\eta=0$ to get
\[
\left(\int_0^T\norm{u(t)}_{q}^{q}dt\right)^{\frac{p}{q}}\leq
C\norm{u}_{\mathcal{H}_p^{\mu}(Q_T)}^p\ .
\]
Item (ii) is then a consequence of the above computations setting $q=\delta$, $\mu=2s-1$, $\theta=p/q=p/\delta$. 

The last assetion (iii) is a consequence of \cite[Theorem 2.6]{CG1}, while last assertion is a byproduct of the embedding in Lemma \ref{trace} and Lemma \ref{inclstat}, since
\[
\H_p^{2s}(Q_T)\hookrightarrow C([0,T];W^{2s-\frac{2s}{p},p}(\T))\hookrightarrow C([0,T];C^{2s-\frac{d+2s}{p}}(\T))
\]

\end{proof}
\begin{lemma}\label{compact}
Let $1<p<\frac{d+2s}{\mu}$, $\mu\in\R$, $\mu>0$. Then, the space $\H_p^\mu(Q_T)$ is compactly embedded onto $L^q(Q_T)$ for $1\leq q<\frac{(d+2s)p}{d+2s-\mu p}$.
\end{lemma}
\begin{proof}
To show the compactness, we restrict to consider the case $\mu\in(0,2s]$, the general case being consequence of the isometry property of the operator $(I-\Delta)^{\frac{\mu}{2}}$ on spaces of Bessel potentials. The idea is to exploit the so-called Aubin-Lions-Simon Lemma. Let $\mu\in\R$ and $0<\mu\leq 2s$ with $p$ satisfying $1<p<\frac{d+2s}{\mu}$. Note first that $H_{p'}^{\mu}(\T)$ is reflexive and separable. Therefore the space $L^p(0,T;(H_{p'}^{\mu}(\T))')$ is isomorphic to $(L^{p'}(0,T;H_{p'}^{\mu}(\T)))'\equiv (\mathbb{H}_{p'}^{\mu}(Q_T))'$. One can easily see that, by definition, $\H_{p}^{\mu}(Q_T)$ is isomorphic to
\begin{equation*}
E:=\{u\in L^p(0,T;H_p^{\mu}(\T)), \partial_t u\in L^p(0,T;(H_{p'}^{2s-\mu}(\T))'\}\ .
\end{equation*}
Note also that $H_p^{\mu}(\T)$ is compactly embedded into $L^p(\T)$ by Lemma \ref{inclstat}-(iv) and $L^p(\T)$ is continuously embedded in $(H_{p'}^{2s-\mu}(\T))'$ since $\mu\leq 2s$. Then, Aubin-Lions-Simon Lemma (see \cite{Simon} and \cite[Proposition III.1.3]{Showalter}) implies that $E$ is compactly embedded into $L^p(Q_T)$. Hence $\mathcal{H}_p^{\mu}(Q_T)$ is compactly embedded in $L^q(Q_T)$ for any $1\leq q\leq p$. Let $u_n$ be a bounded sequence in $\mathcal{H}_p^{\mu}(Q_T)$. By the previous discussion we may extract a subsequence $u_{n_k}$ converging to $u$ strongly in $L^p(Q_T)$. For any $p<q<\frac{(d+2s)p}{d+2s-\mu p}$, arguing by interpolation, we may assert the existence of $0<\theta<1$ such that
\begin{equation*}
\norm{u_{n_k}-u_{n_j}}_{L^q(Q_T)}\leq\norm{u_{n_k}-u_{n_j}}^{\theta}_{L^p(Q_T)}\norm{u_{n_k}-u_{n_j}}^{1-\theta}_{L^{\frac{(d+2s)p}{d+2s-\mu p}}}\rightarrow0
\end{equation*}
as $j,k\rightarrow+\infty$, since $u_{n_k}$ belongs to $\mathcal{H}_p^{\mu}(Q_T)$, which is in turn continuously embedded onto $L^{\frac{(d+2s)p}{d+2s-\mu p}}$ in view of Lemma \ref{embs}, so $u_{n_k}$ converges strongly also in $L^q(Q_T)$.
\end{proof}
We now recall a maximal regularity theorem for fractional heat equations. Consider the problem
\begin{equation}\label{fracreghol}
\begin{cases}
\partial_tu+(-\Delta)^su=f(x,t)&\text{ in }Q_T\ ,\\
u(x,0)=u_0(x)&\text{ in }\T\ .
\end{cases}
\end{equation}
We have the following result for strong solutions to \eqref{fracreghol}, i.e. $u\in \H_q^{2s}$, the equation is solved a.e. and $u(0)$ is meant in the sense of traces.
\begin{thm}\label{fracregtor} Let $p > 1$. Suppose that $u \in \mathcal{H}_p^{\mu}(Q_T)$ solves \eqref{fracreghol}. Then, every strong solution to \eqref{fracreghol} verifies 
\[
\|u\|_{\mathcal{H}_p^{\mu}(Q_T)} \le C(\|f\|_{\mathbb{H}_p^{\mu-2s}(Q_T)} + \|u_0\|_{W^{\mu - 2s/p,p}(\T)}).
\]
where $C > 0$ depends on $d, T, p, s$ (but remains bounded for bounded values of $T$).
\end{thm}
\begin{proof}
The proof is a consequence of well-known results for abstract evolution equations when $\mu=2s$, see e.g. \cite{HP}. The general case can be handled using the isometry of the Bessel operator as in \cite{k3}, and it is proved in \cite{CL}. In particular, in \cite{CL} the proof is provided for stochastic PDEs, which makes necessary the restriction $p>2$. However, for standard PDEs one simply requires $p>1$, as it can be seen in \cite[Lemma 3.2 and Lemma 3.4]{CL}.
\end{proof}

The last part of the section is devoted to present a Sobolev embedding theorem for the parabolic Bessel potential class $\H_p^{2s-1}$ with traces on the hyperplane $t=0$ in $L^1$.
This can be regarded as a nonlocal counterpart of \cite[Proposition A.2]{CG2}. The result is given via the above interpolation theory arguments, although a different proof can be done as in \cite[Appendix A]{CG2} via duality.
\begin{lemma}\label{embL^1frac}
Let $s\in(\frac12,1)$. If $1<\sigma'<(d+2s)/(d+2s-1)$, then $\mathcal{H}_{\sigma'}^{2s-1}(Q_T)$ is continuously embedded onto $L^p(Q_T)$ for
\[
\frac1p=\frac{1}{\sigma'} -\frac{2s-1}{d+2s}.
\] 
Moreover, if $u \in \H_{\sigma'}^{2s-1}(Q_T)$ and $u(\cdot, 0) \in L^1(\T)$, we have
\begin{equation}\label{immL1}
\norm{u}_{L^p(Q_T)}\leq C\big(\norm{u}_{\H_{\sigma'}^{2s-1}(Q_T)}+\norm{u(0)}_{L^1(\T)}\big)\ ,
\end{equation}
where the constant $C$ depends on $d,p,\sigma',T$, but remains bounded for bounded values of $T$.
\end{lemma}

\begin{proof}
The result is a consequence of Lemma \ref{embs}-(i) with $p=\sigma'$, $\mu=2s-1$ and the fact that
\[
\|u(0)\|_{W^{2s-1-2s/\sigma',\sigma'}(\T)}\leq \tilde C\|u(0)\|_{L^1(\T)}
\]
for some positive constant $\tilde C>0$ provided that $\sigma>d+2s$, i.e. $\sigma'<\frac{d+2s}{d+2s-1}(<2)$.
\end{proof}

\begin{lemma}\label{embnik}
Let $\mu>0$ and $1\leq p<\infty$. Then $W^{\mu,p}(\T)\subseteq N^{\mu,p}(\T)$ with continuous embedding. In particular, the space $L^p(I;W^{\mu,p}(\T))\subseteq  L^p(I;N^{\mu,p}(\T))$ with continuous inclusion, where $I\subset\R$. Similarly, we have $H_p^\mu(\T)\subseteq N^{\mu,p}(\T)$ and hence $L^p(I;H^{\mu}_p(\T))\subseteq  L^p(I;N^{\mu,p}(\T))$.
\end{lemma}
\begin{proof}
The embedding $W^{\mu,p}(\T)\subseteq N^{\mu,p}(\T)$ is proved in \cite{Taibleson2},\cite{Stein}, see also \cite[Lemma A.3]{CG5} for the periodic setting. As for $H_p^\mu(\T)\subseteq N^{\mu,p}(\T)$, we use the chain of inclusions $H_p^\mu(\T)\simeq F^\mu_{p2}(\T)\subseteq B^\mu_{p2}(\T)\subseteq B^\mu_{p\infty}(\T)\simeq N^{\mu,p}(\T)$, $F_{pq}^\mu(\T)$ being the periodic Triebel-Lizorkin space, cf \cite{ST}. Here, the first inclusion follows by the embedding $F^s_{p_1q}(\T)\subseteq B^s_{p_0q}(\T)$ valid for $p_0\leq p_1$, $q_0\leq\infty$ and $s\in\R$, cf \cite[Section 3.5.1, Remark 4]{ST}, applied with $s=\mu$, $p_0=p_1=p$ and $q=2$, while the second embedding is a consequence of the inclusion $B^s_{pq_0}(\T)\subseteq B^s_{pq_1}(\T)$ for $p\leq\infty$, $q_0\leq q_1\leq\infty$, $s\in\R$, see \cite[Section 3.5.1, Remark 4]{ST}, applied with $s=\mu$ and $q_0=2$, $q_1=\infty$. The proof of the equivalences $H_p^\mu(\T)\simeq F^\mu_{p2}(\T)$ and $B^\mu_{p\infty}(\T)\simeq N^{\mu,p}(\T)$ can be found in \cite[Theorem 3.5.4-(iv) and (v)]{ST}.
\end{proof}
\section{Fractional Fokker-Planck equations}\label{sec;ffp}

\subsection{Weak solutions for the fractional Fokker-Planck equation}
This part is devoted to study the following Fokker-Planck equation with fractional diffusion
\begin{equation}\label{fpnonlocal}
\begin{cases}
-\partial_t \rho(x,t)+(-\Delta)^s\rho(x,t)+\dive(b(x,t)\, \rho(x,t))=0&\text{ in }Q_\tau\ ,\\
\rho(x,\tau)=\rho_\tau(x)&\text{ in }\T\ .
\end{cases}
\end{equation}
Note that when the vector field $b(x,t) = - D_pH(x, Du(x,t))$, then \eqref{fpnonlocal} becomes the adjoint equation of the linearization of \eqref{hjb}.
Here, $\tau \in (0,T]$ and $Q_\tau:=\T\times (0,\tau)$. From now on, unless otherwise specified, we will focus on $d>2$. We will consider the following notion of weak solution
\begin{defn}\label{wfkp}
Let $b \in L^\sQ(0,T; L^\sP(\T))$ with $\sP\in(d/(2s-1),\infty)\text{ and }\sQ \in(2s/(2s-1),\infty]$ be such that
\begin{equation}\label{b}
\frac{d}{2s\sP}+\frac{1}{\sQ}< \frac{2s-1}{2s}\ ,
\end{equation}
and $\rho_\tau\in H^{s-1}(\T)$. A (weak) solution $\rho$ to \eqref{fpnonlocal} belongs to $\H_2^{2s-1}(Q_\tau)$ and satisfies
\begin{equation}
\int_0^\tau\int_\T\partial_t\rho\varphi \, dxdt + \iint_{Q_\tau} (-\Delta)^{s-\frac{1}{2}} \rho \, (-\Delta)^{\frac12} \varphi  - b \rho \cdot D\varphi \,dxdt = \int_\T \rho_\tau(x)\varphi(x,\tau)\,dx
\end{equation}
for all $\varphi\in \H_2^1(\T\times(0,\tau])$.
\end{defn}
In particular, the above formulation holds even when test functions are chosen to belong to the class $\H_2^{1}(Q_\tau):=\{\varphi\in L^2(0,\tau;H^1(\T))\ ,\partial_t\varphi\in L^2(0,\tau;H^{-2s+1}(\T))\}$.
We stress out that when $s=1$ the above setting falls within the classical matter described in \cite{BOP,LSU,BCCS} . We remark in passing that $\rho\in \H_2^{2s-1}(Q_\tau)\hookrightarrow C([0,T];(H^{2s-1}(\T),H^{-1}(\T))_{1/2,2})\simeq C([0,T];H^{s-1}(\T))$ in view of the classical abstract trace result \cite[Section XVIII.3 eq. (1.61)]{DL}. 
\begin{rem}\label{intparts}
We point out that time-integration by parts
\begin{equation}\label{parts}
\iint_{Q_{\tau}} \varphi\partial_t\rho+\iint_{Q_{\tau}} \partial_t\varphi \rho\,dxdt=\int_\T \varphi(x,\tau)\rho(x,\tau)\,dx-\int_\T \varphi(x,\omega)\rho(x,\omega)\,dx
\end{equation}
holds, where duality pairings are hidden here. To prove this fact, one represents $\H_2^{2s-1}(Q_{\tau})$ as
\[
\H_2^{2s-1}(Q_{\tau})=\{u\in L^2(0,\tau;H^{2s-1}(\T))\ ,\partial_tu\in L^2(0,\tau;H^{-1}(\T))\}
\] 
which coincides with the space $W(0,\tau,H^{2s-1}(\T),H^{-1}(\T))$ defined in \cite[Chapter XVIII, Section 3]{DL}. Then, one uses that $C_0^\infty([0,\tau];H^{2s-1}(\T))$ is dense in $\H_2^{2s-1}(\T)$, the embedding $\H_2^{2s-1}(Q_\tau)\hookrightarrow C([0,T];(H^{2s-1}(\T),H^{-1}(\T))_{1/2,2})\simeq C([0,T];H^{s-1}(\T))$ to give sense to the traces and the fact that \eqref{parts} is true for $\varphi,\rho\in C_0^\infty([0,\tau];H^{2s-1}(\T))$ by the theory of integration and derivation in Banach spaces. In this setting, it is sufficient to have $H^{2s-1}(\T)\hookrightarrow H^{-1}(\T)$, with $H^{2s-1}(\T)$ dense in $H^{-1}(\T)$, cf \cite[Proposition 3.3]{Magenes}, as described in \cite{DL}.
\end{rem}

Throughout this section we will assume that
\begin{equation}\label{rhoassfrac}
\rho_\tau\in H^{s-1}(\T), \quad \rho_\tau\geq0, \quad \text{and} \quad \int_\T \rho_\tau(x)\,dx=1\ .
\end{equation}
We further observe that since $s>1/2$ we have $\rho\in\H_2^{2s-1}$ and $\H_2^{2s-1}\hookrightarrow L^{\frac{2(d+2s)}{d+2-2s}}\hookrightarrow L^2\hookrightarrow L^1$ and hence $\rho(t)\in L^1(\T)$ for a.e. $t$. Therefore, by using $\varphi\equiv1$ as a test function one obtains $\int_\T\rho(t)=1$ for $t\in(0,T)$. \\
 \begin{rem}
 Note that on $\R^d\times(0,T)$ it is easy to check that the equation 
\[
\partial_t\rho+(-\Delta)^s\rho+\mathrm{div}(b(x,t)\rho)=0
\]
 is invariant under the scalings
\[
\rho_\lambda(x,t):=\rho(\lambda x,\lambda^{2s} t)\text{ and }b_\lambda(x,t):=\lambda^{2s-1} b(\lambda x,\lambda^{2s} t)\ .
\]
Therefore, when looking the equation at small scales, for $s\in(1/2,1)$ one has to check the effect of the scaling on the Lebesgue norm of the velocity filed. In such case, the subcritical space turns out to be the mixed space $L^\sQ(L^\sP)$ when the exponents $\sP\geq d/(2s-1)$ and $\sQ\geq2s/(2s-1)$ fulfill the condition
\[
\frac{d}{2s\sP}+\frac{1}{\sQ}\leq \frac{2s-1}{2s}\ ,
\]
which can be seen as the fractional counterpart of the classical Aronson-Serrin interpolated condition for viscous problems with unbounded coefficients \cite{LSU,BCCS,BOP} mentioned in the introduction. This condition allows to give a distributional sense to the transport term. Indeed, for $\varphi\in\H_2^1$, $\rho\in\H_2^{2s-1}$ and $\sP=\sQ$, we have by H\"older's inequality
\[
\iint \mathrm{div}(b(x,t)\rho)\varphi=- \iint b(x,t)\rho\cdot D\varphi\leq \|b\|_{L^{\frac{d+2s}{2s-1}}}\|\rho\|_{L^{\frac{2(d+2s)}{d+2-2s}}}\|D\varphi\|_{L^2}\lesssim \|b\|_{L^{\frac{d+2s}{2s-1}}}\|\rho\|_{\H_2^{2s-1}}\|D\varphi\|_{L^2}\ .
\]
 \end{rem}
Classical Fokker-Planck equations with low regularity assumptions on the drift have been studied in \cite{Porr,MPR,BKRS,CG2} and references therein. 
\subsection{Existence and integrability estimates}
We premise the following auxiliary result that allows to deduce positivity and uniqueness for the solution $\rho$ to \eqref{fpnonlocal}.
\begin{lemma}\label{adjointunique}
Any weak solution to \eqref{fpnonlocal} satisfies 
\begin{equation}\label{dual}
\iint_{Q_\tau}\rho f\,dxdt=\int_\T \rho(\tau)v(\tau)\,dx
\end{equation}
for any $v\in \H_2^1$ solution to \eqref{drift}.
\end{lemma}
\begin{proof}
Let $v$ be a weak solution to the problem
\[
\begin{cases}
\partial_t v+(-\Delta)^s v+b(x,t)\cdot Dv=f(x,t)&\text{ in }Q_\tau\\
v(x,0)=0&\text{ in }\T
\end{cases}
\]
Then, by duality we immediately get \eqref{dual}.
\end{proof}
We now present the main result of this section. Note that our approach is based on maximal regularity arguments, which is a different strategy compared to \cite{BOP}.
\begin{proof}[Proof of Theorem \ref{estadjfrac}]
\textit{Step 1. Existence in the energy space $\H_2^{2s-1}(Q_\tau)$}. We apply Leray-Schauder fixed point theorem for the existence (see \cite[Theorem 11.6]{GT}) on the space $\H_2^{2s-1}(Q_\tau)$.
Consider the map $\mathcal{M}:\H_2^{2s-1}\times [0,1]\to \H_2^{2s-1}(Q_\tau)$ defined by $m\longmapsto \rho=\mathcal{M}[m;\sigma]$ given by solving the following parametrized PDE
\[
-\partial_t\rho+(-\Delta)^s\rho=\sigma\mathrm{div}(b(x,t)m)\text{ in }Q_\tau\ ,\rho(x,\tau)=\sigma\rho_\tau(x)\text{ in }\T\ .
\]
Note that $\mathcal{M}[m;0]=0$ by standard results for fractional heat equations. We first show that it is well-defined. We start with the case $\sP=\sQ$ (whence condition \eqref{b} becomes $\sP>\frac{d+2s}{2s-1}$). 
By parabolic Calder\'on-Zygmund regularity theory (cf Theorem \ref{fracregtor}) we have
\begin{multline}\label{apriorileray}
\|\rho\|_{\H_2^{2s-1}(Q_\tau)}\leq C(\sigma\|bm\|_{L^2(Q_\tau)}+\sigma\|\rho_\tau\|_{H^{s-1}(\T)})\\
\leq C(\|b\|_{L^\sP(Q_\tau)}\|m\|_{L^{\frac{2\sP}{\sP-2}}(Q_\tau)}+\|\rho_\tau\|_{H^{s-1}(\T)})\ .
\end{multline}
Now, note that
\[
1< \frac{2\sP}{\sP-2}< \frac{2(d+2s)}{d+2-2s}\ .
\]
We then argue by interpolation, exploit the embedding of $\H_2^{2s-1}(Q_\tau)\hookrightarrow L^\frac{2(d+2s)}{d+2-2s}(Q_\tau)$ in Lemma \ref{embs} and the fact that $m\in L^1(Q_\tau)$ to show, applying also Young's inequality,
\begin{multline*}
\|m\|_{L^\frac{2\sP}{\sP-2}(Q_\tau)}\leq C_1\|m\|_{L^1(Q_\tau)}^\theta\|m\|_{L^\frac{2(d+2s)}{d+2-2s}(Q_\tau)}^{1-\theta}=C_1\tau^\theta\|m\|_{L^\frac{2(d+2s)}{d+2-2s}(Q_\tau)}^{1-\theta}\leq C_2
+\eps\|m\|_{\H_2^{2s-1}(Q_\tau)}
\end{multline*}
for some $\theta\in(0,1)$, $\eps>0$. Then, for $\eps=1/2$ we have
\[
\|\rho\|_{\H_2^{2s-1}(Q_\tau)}\leq C_2(\|b\|_{L^\sP(Q_\tau)}+\|\rho_\tau\|_{H^{s-1}(\T)})+\frac12\|m\|_{\H_2^{2s-1}(Q_\tau)}\ .
\]
This shows that $\mathcal{M}$ is well-defined from $\H_2^{2s-1}(Q_\tau)$ into itself, since $m\in\H_2^{2s-1}(Q_\tau)$. Moreover, if $\rho\in \H_2^{2s-1}(Q_\tau)$ and $\sigma\in[0,1]$ is a fixed point of the map $\rho=\mathcal{M}[\rho;\sigma]$ we have that $\rho\in \H_2^{2s-1}(Q_\tau)$ is a solution of \eqref{fpnonlocal} and the a priori estimate \eqref{apriorileray} carry through uniformly on $\sigma\in[0,1]$. Thus, we obtain the existence of a constant $M>0$ depending only on the data (namely $\|b\|_{L^\sP(Q_\tau)}$, $\rho_\tau$, $T,s$) such that
\[
\|\rho\|_{\H_2^{2s-1}(Q_\tau)}\leq M\ .
\]
We finally show that the map $\mathcal{M}$ is compact. Let $m_n$ be a bounded sequence in $\H_2^{2s-1}(Q_\tau)$ and let $\rho_n=\mathcal{M}[m_n;\sigma]$ with $\rho_n(\tau)=\sigma\rho_\tau$. Since $|b|m_n\in L^2(Q_\tau)$ we have that $\mathrm{div}(bm_n)\in\mathbb{H}_2^{-1}(Q_\tau)$ and hence by Theorem \ref{fracregtor} we deduce $\rho_n\in \H_2^{2s-1}(Q_\tau)$. By the compactness of $\H_2^{2s-1}$ onto $L^2(Q_\tau)$ (cf Lemma \ref{compact}), which is ensured by the restriction $s>1/2$, we have that, along a subsequence, $\rho_n$ converges strongly in $L^2(Q_\tau)$ to $\rho$ and $(-\Delta)^{s-1/2}\rho_n$ converges weakly to $(-\Delta)^{s-1/2}\rho$ in $L^2(Q_\tau)$. Moreover, $\rho$ solves the same problem as $\rho_n$ given the couple $(m,\sigma)$. We use $(-\Delta)^{s-1}(\rho_n-\rho)\in\H_2^{1}(Q_\tau)$ as admissible test function in the weak formulation of the equation satisfied by $\rho_n$, together with the fact that
\begin{multline*}
-\iint_{Q_t}\partial_t(\rho_n-\rho)(-\Delta)^{s-1}(\rho_n-\rho)\,dxdt=-\frac12\int_t^\tau\int_\T\partial_t[(-\Delta)^{\frac{s-1}{2}}(\rho_n-\rho)]^2\,dxdt\\
=-\frac12\int_\T [(-\Delta)^{\frac{s-1}{2}}(\rho_n-\rho)]^2(\tau)\,dx+\frac12\int_\T [(-\Delta)^{\frac{s-1}{2}}(\rho_n-\rho)]^2(t)\,dx\ ,
\end{multline*}
 to conclude 
\begin{multline*}
\iint_{Q_\tau}|(-\Delta)^{s-\frac{1}{2}}(\rho_n-\rho)|^2\,dxdt\\
\leq C\iint_{Q_\tau}|b|m_n||(-\Delta)^{s-\frac12}(\rho_n-\rho)|\,dxdt-\iint_{Q_\tau}(-\Delta)^{s}\rho(-\Delta)^{s-1}(\rho_n-\rho)\,dxdt\\
+\iint_{Q_\tau}\partial_t\rho(-\Delta)^{s-1}(\rho_n-\rho)\,dxdt\ .
\end{multline*}
Since $|b|m_n\in L^2(Q_\tau)$ and $(-\Delta)^{s-1/2}\rho_n$ converges weakly to $(-\Delta)^{s-1/2}\rho$ in $L^2(Q_\tau)$ the first term on the right-hand side of the above inequality converges to 0. Similarly, since $\partial_t\rho\in\mathbb{H}_2^{-1}(Q_\tau)$ and exploiting again the weak convergence of $(-\Delta)^{s-1/2}\rho_n$ in $L^2(Q_\tau)$ the third term goes to 0. Similar motivations provide the convergence of the second term. This shows that $(-\Delta)^{s-1/2}\rho_n$ converges strongly to $(-\Delta)^{s-1/2}\rho$ in $L^2(Q_\tau)$.\\
Finally, to show the strong convergence of $\partial_t\rho_n$ to $\partial_t\rho$ in $\mathbb{H}_2^{-1}(Q_\tau)$ we argue by duality. For every $\varphi\in \mathbb{H}_2^1(Q_\tau)$ we have
\begin{multline*}
\left|\iint_{Q_\tau} \partial_t(\rho_n-\rho)\varphi\,dxdt\right|\leq \left|\iint_{Q_\tau}(-\Delta)^{s}(\rho_n-\rho)\varphi\,dxdt\right|+\left|\iint_{Q_\tau}\mathrm{div}(b(\rho_n-\rho))\varphi\,dxdt\right|\\
\leq C\iint_{Q_\tau}|(-\Delta)^{s-\frac12}(\rho_n-\rho)||D\varphi|\,dxdt+\iint_{Q_\tau}|\rho_n-\rho||b||D\varphi|\,dxdt\ ,
\end{multline*}
 which yields the strong convergence of $\partial_t\rho_n$ to $\partial_t\rho$ in $\mathbb{H}_2^{-1}(Q_\tau)$ in view of the previous claims. \\The general case $\sP\neq\sQ$ can be dealt with similarly. Indeed, in the borderline case $\sQ=\infty$, we observe that
 \[
 \|bm\|_{L^2(Q_\tau)}\leq c\||b|\|_{L^\infty(0,\tau;L^{\sP}(\T))}\|m\|_{L^2(0,\tau;L^\frac{2\sP}{\sP-2}(\T))}\ .
 \]
We then observe that
\[
1< \frac{2\sP}{\sP-2}<\frac{2d}{d-2(2s-1)}\ ,
\]
which yields by interpolation for $\sP>\frac{d}{2s-1}$ the inequality
\[
\|m\|_{L^2(0,\tau;L^\frac{2\sP}{\sP-2}(\T))}\leq C\|m\|_{L^2(0,\tau;L^\frac{2d}{d-2(2s-1)}(\T))}^{\theta}
\]
for a.e. $t\in(0,\tau)$. Using the Sobolev embedding $H_2^{2s-1}(\T)\hookrightarrow L^\frac{2d}{d-2(2s-1)}(\T))$ we conclude
\[
\|m\|_{L^2(0,\tau;L^\frac{2\sP}{\sP-2}(\T))}\leq C\|m\|_{\H_2^{2s-1}(Q_\tau)}^{\theta}
\]
and then proceed as above. When $\sP,\sQ$ are finite, we have
\[
 \|bm\|_{L^2(Q_\tau)}\leq \||b|\|_{L^\sQ(0,\tau;L^{\sP}(\T))}\|m\|_{L^\frac{2\sQ}{\sQ-2}(0,\tau;L^\frac{2\sP}{\sP-2}(\T))}\ .
 \]
We now use interpolation with $\eta,\delta,\zeta$ (cf \cite[Lemma 1]{AS}) satisfying
\[
\frac{\sQ-2}{2\sQ}=\frac{1-\theta}{\zeta}+\frac{\theta(\eta-2)}{2\eta}\ ,
\]
\[
\frac{\sP-2}{2\sP}=1-\theta+\frac{\theta(\delta-2)}{2\delta}\ .
\]
for $\theta\in(0,1)$, $\eta<\sQ$, $\delta<\sP$.
This gives, using that $m\in L^1(\T)$,
\[
\|m\|_{L^\frac{2\sQ}{\sQ-2}(0,\tau;L^\frac{2\sP}{\sP-2}(\T))}\leq C\|m\|_{L^\frac{2\eta}{\eta-2}(0,\tau;L^\frac{2\delta}{\delta-2}(\T))}^{\theta}
\]
We now exploit the mixed-norm embedding in \cite[Proposition 2.11]{CG1} (applied with $p=2$, $q=\frac{2\delta}{\delta-2}$, $\theta=\frac{\eta-2}{\eta}$, $\mu=2s-1$) to conclude
\[
\|m\|_{L^\frac{2\eta}{\eta-2}(0,\tau;L^\frac{2\delta}{\delta-2}(\T))}^\theta\leq C\|m\|_{\H_2^{2s-1}(Q_\tau)}^\theta
\]
provided that
\[
\frac{d}{2s\delta}+\frac{1}{\eta}<\frac{2s-1}{2s}\ ,
\]
i.e. when \eqref{b} holds for $\sP,\sQ$.\\
\par	\smallskip

\textit{Step 2. A priori estimates via Duhamel's formula}. The proof we are going to present can be made rigorous by regularization (cf \cite[Lemma 2.3]{PorrUMI}), using Duhamel's formula for the regularized PDE and then passing to the limit. The approach is inspired by \cite{BCCS} and it has been also recently implemented in \cite[Lemma A.3]{CPorr} to get estimates in mixed Lebesgue scales and in \cite[Lemma A.3]{CGhilli}.\\
We claim that there exists $t^*\in(0,\tau]$ independently of $\rho_\tau\in L^p(\T)$ such that
\[
\|\rho(\cdot,t)\|_{L^p(\T)}\leq C_2\|\rho_\tau\|_{L^p(\T)}\text{ for all $t\in[t^*,\tau]$}
\]
for some $C_2>0$. Set $\tilde\rho(\cdot,t):=\rho(\cdot,\tau-t)$ for all $t\in[0,\tau]$ and use Duhamel's formula to represent the solution of the (forward) equation as
\[
\tilde\rho(t)=\mathcal{T}_t\rho_\tau-\int_0^t \mathcal{T}_{t-\omega}\mathrm{div}(b\tilde \rho)(\cdot,\omega)d\omega\ .
\]
where $\mathcal{T}_t=e^{-t(-\Delta)^s}$. We have
\begin{multline*}
\|\tilde \rho(t)\|_{L^p(\T)}\leq \|\mathcal{T}_t\rho_\tau\|_{L^p(\T)}+\left\|\int_0^t \mathcal{T}_{t-\omega}\mathrm{div}(b\tilde \rho)(\cdot,\omega)d\omega\right\|_{L^p(\T)}\\
\leq \|\rho_\tau\|_{L^p(\T)}+\int_0^t(t-\omega)^{-\frac{d}{2s}(\frac1a-\frac1p)-\frac{1}{2s}}\|\mathrm{div}(b\tilde\rho)(\cdot,\omega)\|_{H_b^{-1}(\T)}d\omega\\
\leq \|\rho_\tau\|_{L^p(\T)}+\int_0^t(t-\omega)^{-\frac{d}{2s}(\frac1a-\frac1p)-\frac{1}{2s}}\|b\tilde\rho(\cdot,\omega)\|_{L^b(\T)}d\omega\ ,
\end{multline*}
where we applied the decay estimates of the fractional heat semigroup among spaces of Bessel potentials
\[
\|\mathcal{T}_tu\|_{L^p(\T)}\leq Ct^{-\frac{d}{2s}(\frac1a-\frac1p)-\frac{1}{2s}}\|u\|_{H_b^{-1}(\T)}
\]
(cf \cite{CG1}). We then use H\"older's inequality to bound the right-hand side of the last inequality with
\begin{multline*}
\|\tilde\rho\|_{L^\infty(0,\tau;L^p(\T))}\int_0^\tau(t-\omega)^{-\frac{d}{2s}(\frac1a-\frac1p)-\frac{1}{2s}}\|b(\cdot,\omega)\|_{L^{\sP}(\T)}d\omega\\
\leq \left(\int_0^t(t-\omega)^{[-\frac{d}{2s}(\frac1a-\frac1p)-\frac{1}{2s}]\sQ'}\right)^{\frac{1}{\sQ'}}\|b\|_{L^\sQ(0,\tau;L^\sP(\T))}\|\tilde\rho\|_{L^\infty(0,\tau;L^p(\T))}
\end{multline*}
where
\[
\frac{1}{a}=\frac{1}{\sP}+\frac{1}{p}\ .
\]
In particular, the above integral term is well-posed provided that
\[
\alpha:=\left(-\frac{d}{2s\sP}-\frac{1}{2s}\right)\sQ'>-1\ ,
\]
which is indeed satisfied precisely when
\begin{equation*}
\frac{d}{2s\mathpzc{P}} + \frac{1}{\mathpzc{Q}} < \frac{2s-1}{2s}\ .
\end{equation*}
Hence
\[
\|\tilde\rho\|_{L^\infty(0,\tau;L^p(\T))}\leq  \|\rho_\tau\|_{L^p(\T)}+C\|b\|_{L^\sQ(0,\tau;L^\sP(\T))}t^{\frac{\alpha+1}{\sQ'}}\|\tilde\rho\|_{L^\infty(0,\tau;L^p(\T))}\ ,
\]
which gives
\[
\|\tilde\rho\|_{L^\infty(0,\tau;L^p(\T))}\leq 2\|\rho_\tau\|_{L^p(\T)}
\]
by taking
\[
t\geq t^*:=\left(\frac{1}{2C\|b\|_{L^\sQ(0,\tau;L^\sP(\T))}}\right)^{\frac{\sQ'}{\alpha+1}}
\]
and hence the validity of the estimate on $[0,t^*]$. Note that $t^*$ does not depend on $\|\rho_\tau\|_{L^p(\T)}$ and hence one can iterate the argument to get the estimate in $[0,\tau]$ as in \cite{BCCS}. 
\par\smallskip

 \textit{Step 3. Positivity and uniqueness}. Positivity and uniqueness follows exploiting Lemma \ref{adjointunique}. In particular, if $\rho_1,\rho_2$ are two solutions of \eqref{fpnonlocal}, by \eqref{dual} we get
 \[
\iint_{Q_\tau}(\rho_1-\rho_2)f\,dxdt=0
 \]
 which implies $\rho_1=\rho_2$ a.e. on $Q_\tau$. Positivity of solutions follows in a similar way.
\end{proof}
\begin{rem}\label{criticalAS}
In Step 1 we can actually reach the threshold 
\[
\frac{d}{2s\sP}+\frac1\sQ=\frac{2s-1}{2s}
\]
by assuming a smallness condition on $\|b\|_{L^\sQ(L^\sP)}$, since interpolation inequalities are no longer available (cf \cite{Stampacchia} for the elliptic viscous case).
\end{rem}

\subsection{Parabolic Bessel regularity}
We finally describe further regularity results that rely on the information $b\in L^k(\rho\,dxdt)$ for some $k>1$, that will be used in the forthcoming sections. We start with the following maximal $L^q$-regularity result for PDEs with divergence-type terms and terminal data in $L^1$. The method of proof we present below has been already used in \cite{CG2,CT,MPR}.
\begin{prop}\label{estFPfrac}
Let $\rho$ be a (non-negative) weak solution to \eqref{fpnonlocal} and 
\begin{equation}\label{condqfrac}
1<\sigma'<\frac{d+2s}{d+2s-1}\ .
\end{equation}
Then, there exists $C>0$, depending on $\sigma',d,T,s$ such that
\begin{equation}\label{estFP1frac}
\|\rho\|_{\mathcal{H}_{\sigma'}^{2s-1}(Q_\tau)}\leq C(\|b\rho\|_{L^{\sigma'}(Q_\tau)}+\|\rho_\tau\|_{L^1(\T)}).
\end{equation}
\end{prop}
Note that $C$ here does not depend on $\tau \in (0, T]$.

\begin{proof}
Let $\rho$ be smooth, the general case follows by an approximation argument. Let $\varphi$ be a smooth test function vanishing at the initial time $\varphi(\cdot,0)=0$. The strategy follows the proof of \cite[Proposition 2.4]{CG2} and it is based on duality arguments. A different proof of the result will be provided in the next Remark \ref{diffproof}, see also Theorem \ref{estFPnew}. Using the weak formulation of \eqref{adj}, we write for $\varphi$ as above 
\begin{equation}\label{weakform}
\iint_{Q_\tau}\rho(\partial_t\varphi+(-\Delta)^s\varphi-b\cdot D\varphi)\,dxdt=\int_\T \rho_\tau(x)\varphi(x,\tau)\,dx\ .
\end{equation}
Let $\delta>0$ and $\psi=\psi_\delta$ be the solution to the forward fractional heat equation
\[
\begin{cases}
\partial_t\psi+(-\Delta)^s\psi=(|(-\Delta)^{s-\frac12}\rho|^2+\delta)^{\frac{\sigma'-2}{2}}(-\Delta)^{s-\frac12}\rho&\text{ in }Q_\tau\ ,\\
\psi(x,0)=0&\text{ in }\T\ .\\
\end{cases}
\]
By maximal $L^\sigma$-regularity, we get
\begin{multline*}
\|\psi\|_{\H_{\sigma}^{2s}(Q_\tau)}\leq C\|(|(-\Delta)^{s-\frac12}\rho|^2+\delta)^{\frac{\sigma'-2}{2}}(-\Delta)^{s-\frac12}\rho\|_{L^{\sigma}(Q_\tau)}\\\leq C\||(-\Delta)^{s-\frac12}\rho|^{\sigma'-1}\|_{L^\sigma(Q_\tau)}=C\|(-\Delta)^{s-\frac12}\rho\|_{L^{\sigma'}(Q_\tau)}^{\sigma'-1}\ .
\end{multline*}
We take $\varphi=(-\Delta)^{s-\frac12}\psi$ in the weak formulation \eqref{weakform} to see that after integrating by parts
\begin{equation}\label{in}
\iint_{Q_\tau}(|(-\Delta)^{s-\frac12}\rho|^2+\delta)^{\frac{\sigma'-2}{2}}|(-\Delta)^{s-\frac12}\rho|^2\,dxdt\leq \|b\rho\|_{L^{\sigma'}(Q_\tau)}\|D\varphi\|_{L^{\sigma}(Q_\tau)}+\|\rho_\tau\|_{L^{1}(\T)}\|\varphi(\cdot,\tau)\|_{\infty}\ .
\end{equation}
Now, observe that
\begin{multline*}
\|\psi\|_{\H_\sigma^{2s}(Q_T)}\geq C\|\psi\|_{C(W^{2s-2s/\sigma,\sigma}(\T))}=C\|(I-\Delta)^{s-\frac12}\psi\|_{C(W^{1-2s/\sigma,\sigma}(\T))}\\
\geq C\|\varphi(\cdot,\tau)\|_{W^{1-2s/\sigma,\sigma}(\T)}\geq C\|\varphi(\tau)\|_{\infty}
\end{multline*}
when $\sigma>d+2s$ for possibly different positive constants always denoted by $C$. Here, we used the isometry properties of the Bessel potential operator on Sobolev-Slobodeckii scales \cite[Theorem 6.2.7]{BL} and the embeddings in Lemma \ref{inclstatW}.
We then get
\begin{multline*}
\iint_{Q_\tau}(|(-\Delta)^{s-\frac12}\rho|^2+\delta)^{\frac{\sigma'-2}{2}}|(-\Delta)^{s-\frac12}\rho|^2\,dxdt\leq C_1(\|b\rho\|_{L^{\sigma'}(Q_\tau)}+\|\rho_\tau\|_{L^{1}(\T)})\|\psi\|_{\H_\sigma^{2s}(Q_\tau)}\\
\leq C_2(\|b\rho\|_{L^{\sigma'}(Q_\tau)}+\|\rho_\tau\|_{L^{1}(\T)})\||(-\Delta)^{s-\frac12}\rho|\|_{L^{\sigma'}(Q_\tau)}^{\sigma'-1}
\end{multline*}
and let $\delta\to0$ to conclude
\[
\|(-\Delta)^{s-\frac12}\rho\|_{L^{\sigma'}(Q_\tau)}\leq C_2(\|b\rho\|_{L^{\sigma'}(Q_\tau)}+\|\rho_\tau\|_{L^{1}(\T)})\ .
\]
The estimate on $\rho\in L^{\sigma'}(Q_\tau)$ follows by using that of $\|(-\Delta)^{s-\frac12}\rho\|_{L^{\sigma'}(Q_\tau)}$ and the fractional Poincar\'e-Wirtinger inequality in Lemma \ref{PW}. The estimate on the time derivative can be obtained by duality. Indeed, for any $\varphi\in L^\sigma(0,\tau;H^{1}_\sigma(\T))$ we have
\begin{multline*}
\left|\iint_{Q_\tau}\partial_t\rho\varphi\,dxdt\right|=\left|\iint_{Q_\tau}(-\Delta)^{s-\frac12}\rho(-\Delta)^\frac12\varphi\,dxdt\right|+\|b\rho\|_{L^{\sigma'}(Q_\tau)}\|D\varphi\|_{L^{\sigma}(Q_\tau)}\\
\leq C(\|(-\Delta)^{s-\frac12}\rho\|_{L^{\sigma'}(Q_\tau)}+\|b\rho\|_{L^{\sigma'}(Q_\tau)})\|(-\Delta)^{\frac12}\varphi\|_{L^{\sigma}(Q_\tau)}
\end{multline*}
where we used that $W^{1,\sigma}\simeq H^{1}_\sigma$ and H\"older's inequality.
\end{proof}

\begin{rem}\label{diffproof}

A slightly different proof of the above result can be obtained as follows. We observe that $\rho\in\H_2^{2s-1}$ readily implies $\rho\in \H_{\sigma'}^{2s-1}$ for every $1<\sigma'<2$. Let us rewrite equation \eqref{fpnonlocal} as a perturbation of a fractional heat equation
\[
-\partial_t\rho+(-\Delta)^s\rho=\mathrm{div}(b(x,t)\rho)\text{ on }Q_\tau
\]
with terminal data $\rho(x,\tau):=\rho_\tau(x)$ on $\T$. 
By parabolic regularity theory (see Theorem \ref{fracregtor}) $\rho$ enjoys the estimate
\[
\|\rho\|_{\H_{\sigma'}^{2s-1}(Q_\tau)}\leq C(\|b\rho\|_{L^{\sigma'}(Q_\tau)}+\|\rho_\tau\|_{W^{2s-1-2s/\sigma',\sigma'}(\T)})\ .
\]
By exploiting Sobolev embedding for fractional Sobolev spaces in Lemma \ref{inclstatW}, one immediately obtains that 
\[
\|\rho_\tau\|_{W^{2s-1-2s/\sigma',\sigma'}(\T)}\leq C\|\rho_\tau\|_{L^1(\T)}
\]
whenever $1<\sigma'<\frac{d+2s}{d+2s-1}$. Indeed, by \cite[Section 3.5.4]{ST} we have $(W^{\mu,p}(\T))'=W^{-\mu,p'}(\T)$ and thus by definition we get
\begin{multline*}
\|\rho_\tau\|_{W^{2s-1-2s/\sigma',\sigma'}(\T)}=\sup_{\varphi\in W^{2s/\sigma'-2s+1,q}(\T)\ ,\|\varphi\|_{W^{2s/\sigma'-2s+1,\sigma}(\T)}=1}\left|\int_{\T}\rho_\tau\varphi\,dx\right|\\\leq \|\varphi\|_\infty\|\rho_\tau\|_{L^1(\T)}
\leq C\|\varphi\|_{W^{2s/\sigma'-2s+1,\sigma}(\T)}\|\rho_\tau\|_{L^1(\T)}\leq C\|\rho_\tau\|_{L^1(\T)}\ ,
\end{multline*}
where the last inequality is a consequence of the embedding $W^{2s/\sigma'-2s+1,\sigma}(\T)\hookrightarrow C(\T)$ (cf Lemma \ref{inclstatW}-(ii)) when
\[
(2s/\sigma'-2s+1)\sigma>d\ ,
\]
that is $\sigma>d+2s$ or, in other words, when $\sigma'$ satisfies \eqref{condqfrac}. This highlights that the range of $\sigma'$ is imposed by the heat part of the equation.
\end{rem}

The next results asserts fractional Bessel regularity of the fractional Fokker-Planck equation when the trace $\rho_\tau$ belongs to some suitable Lebesgue class. A different proof has been proposed in \cite[Proposition 2.2]{CG5}.
\begin{prop}\label{estFPnew}
Let $\rho$ be a (non-negative) weak solution to \eqref{fpnonlocal}, $\rho_\tau\in L^{p'}(\T)$ and either
\[
\sigma'=\frac{d+2s}{d+2s-1}\text{ with any finite } p>1\ ,
\]
or
\[
\sigma'> \frac{d+2s}{d+2s-1}\text{ with }p'=\frac{d\sigma}{(d+1)\sigma-(d+2s)}\ .
\]
Then, there exists $C>0$, depending on $\sigma',d,T,s$ such that
\begin{equation}\label{estFP1frac}
\|\rho\|_{\mathcal{H}_{\sigma'}^{2s-1}(Q_\tau)}\leq C(\|b\rho\|_{L^{\sigma'}(Q_\tau)}+\|\rho_\tau\|_{L^{p'}(\T)}).
\end{equation}
\end{prop}

\begin{proof}
We can proceed as in Proposition \ref{estFPfrac}, except for the treatment of the term involving $\rho_\tau$. We modify \eqref{in} as
\begin{multline*}
\|\psi\|_{\H_\sigma^{2s}(Q_T)}\geq C\|\psi\|_{C(W^{2s-2s/\sigma,\sigma}(\T))}=C\|(I-\Delta)^{s-\frac12}\psi\|_{C(W^{1-2s/\sigma,\sigma}(\T))}\\
\geq C\|\varphi(\cdot,\tau)\|_{W^{1-2s/\sigma,\sigma}(\T)}\geq C\|\varphi(\tau)\|_{L^{\frac{d\sigma}{d+2s-\sigma}}(\T)}
\end{multline*}
for $\sigma<d+2s$, and
\begin{multline*}
\|\psi\|_{\H_q^{2s}(Q_T)}\geq C\|\psi\|_{C(W^{2s-2s/\sigma,\sigma}(\T))}=C\|(I-\Delta)^{s-\frac12}\psi\|_{C(W^{1-2s/\sigma,\sigma}(\T))}\\
\geq C\|\varphi(\cdot,\tau)\|_{W^{1-2s/\sigma,\sigma}(\T)}\geq C\|\varphi(\tau)\|_{L^{p}(\T)}\ .
\end{multline*}
for any $p\in(1,\infty)$ when $\sigma=d+2s$. 

\end{proof}

As a consequence, the above results yield the proof of Theorem \ref{estFP2frac}.

\begin{proof}[Proof of Theorem \ref{estFP2frac}]

We use Proposition \ref{estFPfrac} and Proposition \ref{estFPnew}, depending on the range of $\sigma'$, and the generalized H\"older's inequality to conclude
\begin{equation}\label{eqqq0}
\|\rho\|_{\mathcal{H}_{\sigma'}^{2s-1}(Q_\tau)}\leq C(\|b\rho^{1/m'}\rho^{1/m}\|_{L^{\sigma'}(Q_\tau)}+1) 
\le C\left(\left(\iint_{Q_\tau} |b|^{m'} \rho \, dxdt\right)^{1/{m'}} \|\rho\|^{1/m}_{L^{\zeta}(Q_\tau)}+1\right),
\end{equation}
for $p > \sigma'$ satisfying
\begin{equation}\label{rdq2frac}
\frac{1}{\sigma'}= \frac{1}{m'} + \frac{1}{m\zeta}.
\end{equation}
Then, by Young's inequality, for all $\eps > 0$
\begin{equation}\label{eqqq1frac}
\|\rho\|_{\mathcal{H}_{\sigma'}^{2s-1}(Q_\tau)} \le C\left( \frac{1}{\eps }\iint_{Q_\tau} |b|^{m'} \rho \, dxdt +  \eps\|\rho\|_{L^{\zeta}(Q_\tau)}+1\right).
\end{equation}
One can verify that the identity $m'=1+\frac{d+2s}{\sigma(2s-1)}$ and \eqref{rdq2frac} yield
\[
\frac1\zeta = \frac{1}{\sigma'} - \frac{2s-1}{d+2s}.
\]
Indeed, \eqref{rdq2frac} gives
\[
\frac1\zeta=\frac{m}{\sigma'}-\frac{m}{m'}=\frac{1}{\sigma'}-\frac{m-1}{\sigma}
\]
and then the definition of $m'$ in \eqref{rdq1frac} yields the conclusion.
The continuous embedding of $\mathcal{H}_{\sigma'}^{2s-1}(Q_\tau)$ in $L^{p}(Q_\tau)$ stated in Lemma \ref{embL^1frac} then implies
\[
 \|\rho\|_{L^{\zeta}(Q_\tau)}  \le C_1 \big( \|\rho\|_{\mathcal{H}_{\sigma'}^{2s-1}(Q_\tau)} + \tau \big)\,,
\]
finally giving
\begin{equation}\label{eqqq2frac}
\|\rho\|_{L^{\zeta}(Q_\tau)} \le CC_1\left( \frac{1}{\eps }\iint_{Q_\tau} |b|^{m'} \rho \, dxdt +  \eps\|\rho\|_{L^{\zeta}(Q_\tau)}+1\right),
\end{equation}
Hence, the term $ \eps\|\rho\|_{L^{\zeta}(Q_\tau)} $ can be absorbed by the left hand side of \eqref{eqqq2frac} by choosing $\eps = (2C C_1)^{-1}$, thus providing the assertion.
\end{proof}
\begin{rem}\label{embdel}
In view of the embedding \[\H_{\sigma'}^{2s-1}\hookrightarrow C([0,\tau];W^{2s-1-\frac{2s}{\sigma'},\sigma'}(\T))\hookrightarrow C([0,\tau];L^{\frac{d\sigma'}{d+2s-(2s-1)\sigma'}}(\T))\ ,\] $p=\frac{d\sigma'}{d+2s-(2s-1)\sigma'}$, estimate \eqref{estFP2frac} implies
\[
\|\rho(t)\|_{L^{p'}(\T)}\leq C\left(\iint_{Q_\tau} |b(x,t)|^{m'} \rho(x,t) \, dxdt  +  \|\rho_\tau\|_{L^{p'}(\T)}\right)
\]
with $p'=\frac{d\sigma}{(d+1)\sigma-(d+2s)}$.
\end{rem}

\begin{cor}\label{corembW}
Let $\rho$ be a nonnegative weak solution to \eqref{fpnonlocal}. Then, there exists $C>0$ depending on $d,q',T,s$ such that for $q<\frac{d+2s}{2s}$
\[
\sup_{t\in[0,\tau]}\|\rho(t)\|_{L^{p'}(\T)}+\|\rho\|_{L^{q'}(Q_\tau)}\leq C\left(\iint_{Q_\tau}|b|^{\frac{d+2s}{(2s-1)q}}\rho\,dxdt+\|\rho_\tau\|_{L^{p'}(\T)}\right)
\]
with $p=\frac{dq}{d+2s-2sq}$, while
\[
\sup_{t\in[0,\tau]}\|\rho(t)\|_{W^{\frac{d(2s-1)}{d+2s}-\frac{2s}{q'},\frac{d+2s}{d+2s+(2s-1)q'}}(\T)}+\|\rho\|_{L^{q'}(Q_\tau)}\leq C\left(\iint_{Q_\tau}|b|^{\frac{d+2s}{(2s-1)q}}\rho\,dxdt+\|\rho_\tau\|_{L^{p'}(\T)}\right)
\]
with $p>1$ arbitrarily large when $q=\frac{d+2s}{2s}$ and $p=\infty$ if $q>\frac{d+2s}{2s}$.
\end{cor}
\begin{proof}
The first estimate is a consequence of Remark \ref{embdel} and Theorem \ref{estFP2frac} applied with $\frac{1}{\sigma'}=\frac{1}{q'}+\frac{2s-1}{d+2s}$ together with the continuous embedding of $\H_{\sigma'}^{2s-1}$ onto $L^{q'}(Q_T)$. The second one exploits also the embedding of $\H_{\sigma'}^{2s-1}$ into $C([0,\tau];W^{\frac{d(2s-1)}{d+2s}-\frac{2s}{q'},\frac{(d+2s)q'}{d+2s+(2s-1)q'}}(\T))$.
\end{proof}
\section{On fractional Hamilton-Jacobi equations}\label{sec;fhjb}
\subsection{On the notions of solutions}

We first provide the following notion of weak solution to \eqref{hjb} we will need to discuss H\"older's regularization effects for \eqref{hjb}. See \cite{CG2} for a similar definition in the viscous case $s=1$.
\begin{defn}\label{wsol} We say that
\begin{itemize}
\item[{\it i)}]  $u$ is a {\it local weak} solution to \eqref{hjb} if for all $0 < \omega < T$
\begin{gather}
\text{$u \in \H_2^1( \T \times (\omega, T)) \cap C(\overline Q_T)$, \quad $H(\cdot, Du) \in L^1(Q_{\omega,T})$, }\label{H12} \\
\text{and $D_p H(\cdot, Du) \in L^{\sP}(Q_{\omega,\tau})$}\label{LqLp} \\
\text{for some $\sP>\frac{d+2s}{2s-1}$, }\label{conditionAS}
\end{gather}
and for all $ 0 < \omega < \tau \le T$, $\varphi \in \H_2^{2s-1}( \T \times (\omega, \tau) ) \cap L^\infty(Q_{\omega,\tau})$
\begin{equation}\label{weake}
\int_\omega^\tau\int_T \partial_t u \varphi\,dxdt +  \iint_{\T \times (\omega, \tau)}  (-\Delta)^\frac12u \,(-\Delta)^{s-\frac12} \varphi + H(x, Du) \varphi \, dxdt =  \iint_{\T \times (\omega, \tau)} f \varphi \,dxdt\ .
\end{equation}
\item[{\it ii)}]   $u$ is a {\it global weak} solution if \eqref{H12}-\eqref{LqLp}-\eqref{conditionAS} hold for all $0 \le \omega< T$, that is, on all $Q_T$ (and therefore, \eqref{weake} is also satisfied up to $\omega = 0$).
\end{itemize}
\end{defn}
\begin{rem}
Under the conditions \eqref{H} on the Hamiltonian, we observe that if \eqref{LqLp} holds, i.e. $D_pH(\cdot,Du)\in L^\sP_{x,t}$, $\sP>\frac{d+2s}{2s-1}$, then condition \eqref{H12} is always satisfied by requiring $\gamma>\frac{d+2s}{d+1}$.
\end{rem}
\begin{defn}\label{strong}
We say that $u\in \H^{2s}_q(Q_T)$, $q>1$, is a strong solution to \eqref{hjb} if the equation is solved for a.e. $(x,t)\in Q_T$ and the initial condition holds in trace sense, i.e. $u(0)\in W^{2s-2s/q,q}(\T)$. 
\end{defn}
\begin{rem}\label{strongweak}
It is immediate to verify by using Sobolev embeddings that whenever $q>\frac{d+2s}{2s}$ and $u\in \H_q^{2s}(Q_T)$, then $u\in \H_2^1(Q_T)$, $u\in C([0,T];W^{2s-2s/q,q}(\T))$ and hence $u$ is a global weak solution. This means that \eqref{weake} is satisfied up to $\omega = 0$.\\
 We note that under the restriction $q>\frac{d+2s}{(2s-1)\gamma'}$ the results in \cite[Proposition 2.11]{CG1} (applied with $p=q$, $\theta=1/\gamma$ and replacing $q$ with $\gamma q$) gives the embedding
 \[
 \H_q^{2s}(Q_T)\hookrightarrow L^{\gamma q}(0,T;W^{1,\gamma q}(\T))\ .
 \]
 This means that $|Du|^{\gamma-1}\in L^\sP_{x,t}$ for $\sP>\frac{d+2s}{2s-1}$. Furthermore, we restrict to consider
 \[
 \gamma>\frac{d+6s-2}{d+2s}\  (>1)
 \]
 so that $u\in L^2(0,T;W^{1,2}(\T))$ (and in particular $u\in \H_2^1(Q_T)$) by parabolic Sobolev embedding and the weak formulation can be safely used. Indeed, embeddings from \cite{CG1} yields that $\H_q^{2s}(Q_T)$ whenever
 \[
 1<2s+\frac{d}{2}-\frac{d+2s(1-q/2)}{q}\implies q>\frac{2(d+2s)}{d+6s-2}\ .
 \]
 In particular, we have
 \[
 \frac{d+2s}{(2s-1)\gamma'}>\frac{2(d+2s)}{d+6s-2}
 \]
 whence $\gamma>\frac{d+6s-2}{d+2s}$. This is consistent with \cite{CG5}, where the condition $\gamma>\frac{d+4}{d+2}$ is needed when $s=1$. We believe that this restriction can be relaxed to
 \[
 \gamma>\frac{d+2s}{d+1}
 \]
 so that $q>\frac{d+2s}{(2s-1)\gamma'}>1$ using techniques from renormalized solutions, cf \cite{Magliocca}. Note also that in the subcritical gradient growth case $\gamma<2s$, solutions are not necessarily continuous when $\frac{d+2s}{(2s-1)\gamma'}\leq q<\frac{d+2s}{2s}$.
\end{rem}
\begin{rem}\label{embin}
By classical embedding properties for Sobolev-Slobodeckii spaces we have the following inclusions
\[
W^{2s-\frac{2s}{q},q}(\T)\hookrightarrow \begin{cases}
C^{2s-\frac{d+2s}{q}}(\T)&\text{ for }q>\frac{d+2s}{2s}\ ,\\
L^p(\T)&\text{ for }p\in[1,\infty)\text{ and }q=\frac{d+2s}{2s}\ ,\\
L^p(\T)&\text{ for }p\in[1,\frac{dq}{d+2s-2s q}]\text{ and }q<\frac{d+2s}{2s}\ .
\end{cases}
\]
\end{rem}

\subsection{Further estimates for the adjoint variable via duality}
First, we prove a simple representation formula for \eqref{hjb} by duality with the adjoint problem
\begin{equation}\label{adj}
\begin{cases}
-\partial_t\rho+(-\Delta)^s\rho-\mathrm{div}(D_pH(x,Du(x,t)\rho(x,t))=0&\text{ in }Q_\tau\ ,\\
\rho(x,\tau)=\rho_\tau&\text{ in }\T
\end{cases}
\end{equation}
where $\rho_\tau\in C^\infty(\T)$, $\rho_\tau\geq0$ with $\|\rho_\tau\|_1=1$.
\begin{lemma}
Let $u$ be a local weak solution to \eqref{hjb}. Assume that $\rho$ is a weak solution to \eqref{adj}. Then, for all $\omega\in(0,\tau)$ we have
\begin{equation}\label{repr}
\int_\T u(x,\tau)\rho_\tau(x)\,dx=\int_\T u(x,\omega)\rho(x,\omega)\,dx+\iint_{Q_{\omega,\tau}}L(x,D_pH(x,Du))\rho\,dxdt+\iint_{Q_{\omega,\tau}}f\rho\,dxdt
\end{equation}
Moreover, if $u$ is either a strong solution in $\H_q^{2s}(Q_T)$ or a global weak solution, the previous identity holds up to $\omega=0$.
\end{lemma}
\begin{proof}
Using $-\rho\in \H_2^{2s-1}(Q_{\omega,\tau})\cap L^\infty(Q_{\omega,\tau})$ as a test function in the weak formulation of \eqref{hjb} and $u\in \H_{2}^{1}(Q_{\omega,\tau})$ as a test function in the corresponding adjoint equation, after summing both expressions we obtain
\begin{multline*}
-\int_\omega^\tau\langle \partial_tu(t),\rho(t)\rangle-\int_\omega^\tau\langle \partial_t\rho(t),u(t)\rangle+\iint_{Q_{\omega,\tau}}(D_pH(x,Du)\cdot Du-H(x,Du))\rho\,dxdt\\
+\iint_{Q_{\omega,\tau}}f\rho\,dxdt=0
\end{multline*}
\end{proof}

We are now ready to prove the crossed integrability bound on $D_pH$ with respect to the density $\rho$. 
\begin{prop}\label{EstimateKfrac}
Let $u$ be a local weak solution to \eqref{hjb} and $\rho$ be a weak solution to \eqref{adj} with $\|\rho_\tau\|_{L^1(\T)}=1$. Then, there exist a positive constant ${C}$ (depending on $C_H$, $\norm{u}_{C(\overline{Q}_T)}$, $\norm{f}_{L^{q}(Q_T)}$, $q, d, T,s$) such that
\begin{equation}\label{estikfrac}
\iint_{Q_\tau}|D_pH(x,Du(x,t))|^{\gamma'}\rho(x,t) \, dxdt \leq {C}\ ,
\end{equation}
\end{prop}

\begin{rem} An immediate consequence of \eqref{estikfrac} is the bound
\begin{equation}\label{estik1frac}
\iint_{Q_\tau}|Du(x,t)|^{\beta}\rho(x,t) \, dxdt \leq {C_\beta}\ \qquad \text{for all $1 \le \beta \le \gamma$}.
\end{equation}
Indeed, by \eqref{H} and $\int_\T\rho(t)=1$ for a.e. $t$, $\iint_{Q_\tau}|Du(x,t)|^{\gamma}\rho(x,t) \, dxdt \leq {C}$, which yields \eqref{estik1frac} for $\beta = \gamma$. For $\beta < \gamma$ it is enough to exploit Young's inequality and $\|\rho(t)\|_{L^1(\T)}=1$.
\end{rem}

\begin{proof}
Rearrange the representation formula \eqref{repr} to get, for $0<\tau_1<\tau<T$,
\begin{multline}\label{ineqK1frac}
\iint_{Q_{\tau_1,\tau}}L(x, D_pH(x,Du))\rho \, dxdt = \int_{\T}u(x,\tau)\rho_{\tau}(x)dx - \int_{\T}u(x,\tau_1)\rho(x,\tau_1)\,dx \\
- \iint_{Q_{\tau_1,\tau}} f\rho\,  dxdt.
\end{multline}

Fix $\eta$ to be determined such that
\[
\eta>\frac{d+2s}{(2s-1)\gamma'}\ .
\]
We use the bounds on the Lagrangian and H\"older's inequality to get
\[
C_L^{-1}\iint_{Q_\tau}|D_pH(x,Du)|^{\gamma'}\rho\leq \iint_{Q_\tau}L(x,D_pH(x,Du)\rho\,dxdt\leq 2\|u\|_{C(\overline{Q}_\tau)}+\|f\|_{L^\eta(Q_\tau)}\|\rho\|_{L^{\eta'}(Q_\tau)}
\]
Let $\sigma$ be such that
\[
\frac{1}{\eta'}=\frac{1}{\sigma'}-\frac{2s-1}{d+2s}\ .
\]
By Lemma \ref{embs} we have
\[
\H_{\sigma'}^{2s-1}(Q_\tau)\hookrightarrow L^{\eta'}(Q_\tau)
\]
Moreover, the choice $\eta>\frac{d+2s}{2s}$ assures that $\sigma'<\frac{d+2s}{d+2s-1}$. Then, in view of Theorem \ref{estFP2frac}-(i) we have
\[
\|\rho\|_{L^{\eta'}(Q_\tau)}\leq C(\|\rho\|_{\H_{\sigma'}^{2s-1}(Q_\tau)}+1)\leq C_1\left(\iint_{Q_\tau}|D_pH(x,Du)|^{m'}\rho\,dxdt+1\right)
\]
for
\[
m'=1+\frac{d+2s}{(2s-1)\sigma}\ .
\]
Thus, we get
\[
C_L^{-1}\iint_{Q_\tau}|D_pH(x,Du)|^{\gamma'}\rho\leq 2\|u\|_{C(\overline{Q}_\tau)}+C_1\|f\|_{L^\eta(Q_\tau)}\left(\iint_{Q_\tau}|D_pH(x,Du)|^{m'}\rho\,dxdt+1\right)\ .
\]
Finally, the right-hand side can be absorbed in the left-hand side when $r'<\gamma'$, i.e.
\[
m'=1+\frac{d+2s}{(2s-1)\sigma}=\frac{d+2s}{(2s-1)\eta}<\gamma'\ .
\]
One then obtain \eqref{estikfrac} by letting $\tau_1\to0$ (note that here constants remain bounded for $\tau_1\in(0,\tau)$).
\end{proof}

The crossed integrability of $D_pH$ against the adjoint variable $\rho$ finally provides the $L^{\sigma'}$ regularity of $(-\Delta)^{s-1/2} \rho$. The next result extends \cite[Corollary 3.4]{CG2} to the fractional framework.

\begin{cor}\label{EstimateK1}
Let $u$ be a local weak solution to \eqref{hjb} and $\rho$ be a weak solution to \eqref{fpnonlocal}. Let $\bar \sigma$ be such that
\[
\bar \sigma > d+2s \quad \text{and} \quad \bar \sigma \ge \frac{d+2s}{(\gamma'-1)(2s-1)}.
\]
Then, there exists a positive constant ${C}$ such that
\begin{equation*}
\|\rho\|_{\mathcal{H}_{\bar \sigma'}^{2s-1}(Q_\tau)}\leq C\ ,
\end{equation*}
where ${C}$ depends in particular on $ C_H$, $\norm{f}_{L^{\bar \sigma}(Q_T)}$, $\norm{u}_{C(\overline{Q}_T)},\eta , d, T,s$ (but not on $\tau, \rho_\tau$), $\bar \sigma>\frac{d+2s}{2s}$.
\end{cor}
\begin{rem}
Note that condition on $\bar \sigma$ can be rewritten as
\[
\bar \sigma>
\begin{cases}
d+2s&\text{ if }\gamma\leq 2s\\
\frac{d+2s}{(\gamma'-1)(2s-1)}&\text{ if }\gamma> 2s\ .
\end{cases}
\]
\end{rem}
\begin{proof} Since $\bar \sigma' < \frac{d+2s}{d+2s-1}$, \eqref{estFP2frac} applies (with $\sigma = \bar \sigma$), yielding
\[
\|\rho\|_{\mathcal{H}_{\bar \sigma'}^{2s-1}(Q_\tau)}\leq C\left( \iint_{Q_\tau} |D_pH(x,Du(x,t))|^{r'} \rho(x,t) \, dxdt + 1\right),
\]
with 
\[
m' = 1 + \frac{d+2s}{\bar \sigma(2s-1)} \le \gamma'.
\]
If $m'=\gamma'$, use Proposition \ref{EstimateKfrac} to conclude. Otherwise, if $m' < \gamma'$
use Young's inequality first to control $\iint |D_pH(x,Du(x,t))|^{r'} \rho \, dxdt$ with $\iint|D_pH(x,Du)|^{\gamma'} \, dxdt + \tau$.
\end{proof}

\subsection{Sup-norm and integral estimates}
We are now ready to prove the sup-norm estimate for global weak solutions to fractional Hamilton-Jacobi equations in terms of $\|u_0\|_{C(\T)}$. 
\begin{proof}[Proof of Theorem \ref{supnorm}]
We argue as in \cite[Proposition 3.7]{CG2}. First, we prove a bound from above for $u$
\[
u(x,\tau)\leq \|u_0\|_{C(\T)}+C\|f\|_{L^q(\T)}
\]
for all $\tau\in(0,T)$, $x\in \T$ and $q>\frac{d+2s}{2s}$. Consider indeed the strong solution to the backward problem
\begin{equation*}
\begin{cases}
-\partial_t\mu+(-\Delta)^s\mu=0&\text{ in }\T\times(0,\tau)\\
\mu(x,\tau)=\mu_\tau(x)&\text{ in }\T
\end{cases}
\end{equation*}
with $\mu_\tau\in C^\infty(\T)$, $\mu_\tau\geq0$ and $\|\mu_\tau\|_{L^1(\T)}=1$. 
%By Proposition \ref{estFP2frac} with $b=0$ we have $\mu\in \H_{q'}^{2s-1}(Q_\tau)$ for $q'<\frac{d+2s}{d+2s-1}$. 
We use $\mu$ as a test function in the weak formulation of the fractional Hamilton-Jacobi equation to deduce
\begin{equation}\label{r}
\int_\T u(x,\tau)\mu_\tau(x)\,dx=\int_\T u(x,0)\mu(x,0)\,dx+\iint_{Q_\tau}f\mu\,dxdt-\iint_{Q_\tau}H(x,Du)\mu\,dxdt\ .
\end{equation}
Using Theorem \ref{estFPfrac} with $b=0$ we find $\|\mu\|_{\H_{\sigma'}^{2s-1}(Q_\tau)}\leq C$ and by Lemma \ref{embs}-(i) we conclude
\[
\|\mu\|_{L^{q'}(Q_\tau)}\leq C\ .
\]
with $q'<\frac{d+2s}{d}$. Then, using the above estimate, H\"older's inequality to the second term of the right-hand side of the above inequality and exploiting that $\|\mu_\tau\|_{L^1(\T)}=1$ one has, choosing $q>\frac{d+2s}{2s}$,
\[
\int_\T u(x,0)\mu(x,0)\,dx+\iint_{Q_\tau}f\mu\,dxdt\leq \|u_0\|_{C(\T)}+C\|f\|_{L^q(Q_\tau)}\ .
\]
By the assumption $H\geq0$ one concludes
\[
\int_\T u(x,\tau)\mu_\tau(x)\,dx\leq \|u_0\|_{C(\T)}+C\|f\|_{L^q(Q_\tau)}
\]
and the claimed estimate from above by duality after passing to the supremum over $\mu_\tau\in L^1(\T)$.\\
To prove the bound from below, we argue using \eqref{repr} with $\omega=0$ and proceed as in Proposition \ref{EstimateKfrac}. Fix some $\eta$ such that
\[
\eta>\frac{d+2s}{(2s-1)\gamma'}\ .
\]
%and note that when $\gamma\leq 2s$ we have $\frac{d+2s}{(2s-1)\gamma'}\leq \frac{d+2s}{2s}$.
We use the bounds on the Lagrangian, the upper bound obtained in Step 1, and H\"older's inequality to get
\begin{multline*}
C_L^{-1}\iint_{Q_\tau}|D_pH(x,Du)|^{\gamma'}\rho\leq \iint_{Q_\tau}L(x,D_pH(x,Du)\rho\,dxdt\leq 2\|u\|_{C(\overline{Q}_\tau)}+\|f\|_{L^\eta(Q_\tau)}\|\rho\|_{L^{\eta'}(Q_\tau)}\\
\leq 2\|u_0\|_{C(\T)}+C\|f\|_{L^q(Q_\tau)}+\|f\|_{L^\eta(Q_\tau)}\|\rho\|_{L^{\eta'}(Q_\tau)}
\end{multline*}
The strategy of Proposition \ref{EstimateKfrac} provides a bound on $\iint_{Q_\tau}|D_pH(x,Du)|^{\gamma'}\rho$ and thus on $\|\rho\|_{L^{\eta'}(Q_\tau)}$, depending on $\|u_0\|_{C(\T)}$.  Going back to \eqref{r} we have
\[
\int_\T u(x,\tau)\rho_\tau\,dx\geq \int_\T u(x,0)\rho(x,0)-C_L\iint_{Q_\tau}\rho\,dxdt+\iint_{Q_\tau}f\rho\,dxdt
\]
Since $\iint f\rho$ can be bounded from below by H\"older's inequality we get
\[
\int_\T u(x,\tau)\rho_\tau(x)\,dx\geq -\|u(\cdot,0)\|_{C(\T)}-C_L\tau-C
\]
Since $\rho_\tau$ can be arbitrarily chosen so that $\|\rho_\tau\|_{L^1(\T)}=1$, we conclude the desired result.
\end{proof}
To deduce integral estimates we use a $L^p$ version of the adjoint method. We recall that integral estimates for parabolic viscous Hamilton-Jacobi equations have been already obtained in \cite{CG5} using the same method, \cite{GomesBook}, see also \cite[Theorem 3.1]{CGPT} for degenerate problems.
\begin{rem}
Results in Theorem \ref{integest} can be regarded as an a priori estimates since \[\H_q^{2s}\hookrightarrow C([0,T];W^{2s-2s/q,q}(\T))\hookrightarrow C([0,T];L^p(\T))\] with $p=\frac{dq}{d+2s-2sq}$.
\end{rem}
\begin{proof}[Proof of Theorem \ref{integest}]
We denote by $T_k(\omega)=\max\{-k,\min\{k,\omega\}\}$ the truncation at level $k>0$, $u^+=\max\{u,0\}$ and $u^-=(-u)^+$.\\
\textbf{Step 1}. We first prove that
\begin{equation}\label{lp}
\|u^+(\cdot,\tau)\|_{L^{p}(\T)}\leq C(\|u_0^+\|_{L^{p}(\T)}+\|f\|_{L^q(Q_\tau)})
\end{equation}
for all $\tau\in(0,T)$ and $x\in\T$, $p\in[1,\infty)$. For $k>0$ consider the weak nonnegative solution to the following backward problem for the fractional heat equation
\[
\begin{cases}
-\partial_t\mu(x,t)+(-\Delta)^s \mu(x,t)=0&\text{ in }Q_\tau\ ,\\
\mu(x,\tau)=\mu_\tau(x)&\text{ in }\T\ ,
\end{cases}
\]
with $\mu_\tau(x)=\frac{[T_k(u^+(x,\tau)]^{p-1}}{\|u^+(\tau)\|_p^{p-1}}$. This truncation argument is needed to ensure the existence of energy solutions. First, observe that $\|\mu_\tau\|_{p'}\leq 1$. By Corollary \ref{corembW} with $b\equiv0$ we deduce
\[
\|\mu\|_{L^{q'}(Q_\tau)}\leq C\ ,
\]
for $\frac{d+2s}{(2s-1)\gamma'}<q<\frac{d+2s}{2s}$ and $p=\frac{dq}{d+2s-2sq}$ or $q=\frac{d+2s}{2s}$ and any $p>1$, with $C$ not depending on $k$. By parabolic Kato's inequality, cf \cite[Theorem 34]{Leonori}, $u^+$ is a subsolution to
\[
\partial_tu^+(x,t)+(-\Delta)^su^+(x,t)\leq[f(x,t)-H(x,Du(x,t))]\chi_{\{u>0\}}\text{ in }Q_\tau\ ,
\]
where $\chi_A$ denotes the indicator function of a given set $A$.
Using $\mu$ as a test function in the weak formulation of the above equation we show that
\[
\int_\T u^+(x,\tau)\mu_\tau(x)\,dx\leq \int_\T u_0^+(x)\mu(x,0)\,dx+\iint_{Q_\tau\cap \{u>0\}}f\mu\,dxdt-\iint_{Q_\tau\cap \{u>0\}}H(x,Du)\mu \,dxdt\ .
\]
We apply H\"older's inequality to the second term of the right-hand side  of the above equality, the fact that $H\geq0$ and the fact that the fractional heat equation preserves the $L^p$ norms, i.e. $\|\mu(t)\|_{L^{p'}(\T)}\leq 1$ for all $t\in[0,\tau]$, to get, after sending $k\to\infty$,
\[
\|u^+(\tau)\|_{L^p(\T)}\leq C(\|u_0^+\|_{L^p(\T)}+\|f^+\|_{L^q(Q_\tau)})\ .
\]
\textbf{Step 2}.  To prove the bound on the negative part, consider for $k>0$ the solution $\rho=\rho_k$ to the adjoint problem
\[
\begin{cases}
-\partial_t\rho(x,t)+(-\Delta)^s \rho(x,t)+\mathrm{div}(D_pH(x,Du)\chi_{\{u<0\}}\rho)&\text{ in }Q_\tau\ ,\\
\rho(x,\tau)=\mu_\tau(x)&\text{ in }\T\ .
\end{cases}
\]
First, by Corollary \ref{corembW} we have
\begin{equation}\label{bo1}
\|\rho(0)\|_{L^{p'}(\T)}+\|\rho\|_{L^{q'}(Q_\tau)}\leq C\left(\iint_{Q_\tau}|D_pH(x,Du)|^{\frac{d+2s}{(2s-1)q}}\chi_{\{u<0\}}\rho\,dxdt+\|\rho_\tau\|_{L^{p'}(\T)}\right)
\end{equation}
with $C$ not depending on $k$. We note again that in view of the parabolic Kato's inequality $u^-$ is a weak subsolution to
\[
\partial_tu^-(x,t)+(-\Delta)^su^-(x,t)\leq[-f(x,t)+H(x,Du(x,t))]\chi_{\{u<0\}}\text{ in }Q_\tau\ .
\]
Owing to the representation formula we get
\begin{multline*}
\int_\T u^-(\tau)\rho(\tau)\,dx+\iint_{Q_\tau}[-D_pH(x,Du)\cdot Du^- -H(x,Du)]\chi_{\{u<0\}}\rho\,dxdt\\
\leq \int_\T u_0^-\rho(0)\,dx-\iint_{Q_\tau\cap \{u<0\}}f\rho\ .
\end{multline*}
By the assumption on the Lagrangian we get
\begin{multline*}
[-D_pH(x,Du)\cdot Du^--H(x,Du)]\chi_{\{u<0\}}=L(x,D_pH(x,-Du^-))\chi_{\{u<0\}}\\
\geq [C_L^{-1}|D_pH(x,Du)|^{\gamma'}-C_L]\chi_{\{u<0\}}\ .
\end{multline*}
Then, we have

\begin{multline*}
\int_\T u^-(\tau)\rho(\tau)\,dx+C_L^{-1}\iint_{Q_\tau}|D_pH(x,Du)|^{\gamma'}\rho\,dxdt-C_L\iint_{Q_\tau}\chi_{\{u<0\}}\rho\,dxdt\\
\leq \|u_0^-\|_{L^p(\T)}\|\rho(0)\|_{L^{p'}(\T)}+\|\rho\|_{L^{q'}(Q_\tau)}\|f^-\|_{L^q(Q_\tau)}\ .
\end{multline*}

Then, we get
\begin{multline}\label{bo2}
\int_\T u^-(\tau)\rho(\tau)\,dx+C_L^{-1}\iint_{Q_\tau}|D_pH(x,Du)|^{\gamma'}\rho\,dxdt-C_L\iint_{Q_\tau}\chi_{\{u<0\}}\rho\,dxdt\\
\leq C\|u_0^-\|_{L^p(\T)}\|\rho(0)\|_{L^{p'}(\T)}+\|\rho\|_{L^{q'}(Q_\tau)}\|f^-\|_{L^q(Q_\tau)}\ .
\end{multline}
The term $\|\rho(0)\|_{L^{p'}(\T)}$ can be bounded owing to Corollary \ref{corembW}, so using \eqref{bo1} we conclude
\begin{multline*}
\int_\T u^-(\tau)\rho(\tau)\,dx+C_L^{-1}\iint_{Q_\tau}|D_pH(x,Du)|^{\gamma'}\rho\,dxdt\\
\leq C(\|u_0^-\|_{L^{p}(\T)}+\|f^-\|_{L^q(Q_\tau)})\left(\iint_{Q_\tau}|D_pH(x,Du)|^{\frac{d+2s}{(2s-1)q}}\chi_{\{u<0\}}\rho\,dxdt+1\right)\\
+C_L\iint_{Q_\tau}\chi_{\{u<0\}}\rho\,dxdt\ .
\end{multline*}
Since $q>\frac{d+2s}{(2s-1)\gamma'}$ one can use Young's inequality and that $\int\rho(t)\,dx\leq \|\mu_\tau\|_{p'}\leq 1$ to conclude, after letting $k\to\infty$,
\begin{multline}
\int_\T u^-(\tau)\rho(\tau)\,dx \leq C(\|u_0^-\|_{L^{p}(\T)}+\|f^-\|_{L^q(Q_\tau)})\\
+CC_L\tau((\|u_0^-\|_{L^{p}(\T)}+\|f^-\|_{L^q(Q_\tau)})^{\frac{\gamma'q(2s-1)}{\gamma'q(2s-1)-(d+2s)}}+C_L\tau\ .
\end{multline}
Combining Step 1 and Step 2 we conclude
\[
\|u\|_{L^\infty(0,\tau;L^p(\T))}\leq C
\]
for $p$ as above.
\end{proof}
\begin{rem}
We underline that knowing a well-posedness result, together with integrability estimates, for the adjoint equation \eqref{fpnonlocal} when $b\in L^{\frac{d+2s}{2s-1}}_{x,t}$ would allow to give an estimate of $u\in L^{d\frac{\gamma-1}{2s-\gamma}}_{x,t}$ when $\gamma<2s$ under finer properties of the data, see e.g. \cite[Remark 3.6]{CG5}. This would be also consistent with results in \cite{Magliocca}. However, such properties of the adjoint variable in the borderline case $b\in L^{\frac{d+2s}{2s-1}}_{x,t}$ are at this stage an open problem, cf \cite{LSU,BOP,BCCS} for the case $s=1$.
\end{rem}
\subsection{H\"older regularity results}
We are now in position to prove H\"older bounds for solutions to the fractional Hamilton-Jacobi equation \eqref{hjb} as in Theorem \ref{mainholder}.

\begin{proof}[Proof of Theorem \ref{mainholder}]
Since $H$ is convex and superlinear, we write for a.e. $(x,t)\in Q_T$
\[
H(x,Du(x,t))=\sup_{\nu\in\R^d}\{\nu\cdot Du(x,t)-L(x,\nu)\}\ .
\]
Let $0<\omega<\tau<T$. By the weak formulation of \eqref{hjb} we obtain
\begin{multline}\label{HJoptimalLOC}
\int_\omega^\tau \langle \partial_t u(t), \varphi(t) \rangle dt +  \iint_{Q_{\omega,\tau}}  (-\Delta)^\frac12 u(x, t) \, (-\Delta)^{s-\frac12} \varphi(x, t) + [\Theta(x, t) \cdot Du(x, t)-L(x,\Theta(x, t))] \varphi \, dxdt \\
\le  \iint_{Q_{\omega,\tau}} f(x, t) \varphi(x, t) \,dxdt
\end{multline}
for all test functions $\varphi \in  \H_2^{2s-1}(Q_{\omega,\tau}) \cap L^\infty(Q_{\omega,\tau})$ and measurable $\Theta : Q_{\omega,\tau} \to \R^d$ such that $L(\cdot, \Theta(\cdot, \cdot)) \in L^1(Q_{\omega,\tau})$ and $\Theta\cdot Du \in  L^1(Q_{\omega,\tau})$. Note that the previous inequality becomes an equality if $\Theta(x,t) = D_pH(x,Du(x,t))$ in $Q_{\omega,\tau}$.

We fix $\rho_\tau\in C^\infty(\T)$, $\|\rho_\tau\|_{L^1(\T)}=1$ and $\rho_\tau\geq0$. Set \[w(x,t)=\eta(t) u(x,t),\] 
where $\eta\in C_0^\infty((0,T])$ is a smooth function such that $0\leq \eta(t)\leq1$ for all $t$.
 
 Use now \eqref{HJoptimalLOC} with $\Xi(x,t) = D_pH(x,Du(x,t))$ and $ \varphi = \eta \rho \in  \H_2^{2s-1}(Q_{\omega,\tau}) \cap L^\infty(Q_{\omega,\tau})$, where $\rho$ is the adjoint variable (i.e. the weak solution to \eqref{adj}) to find
\begin{multline}\label{loc1}
\int_\omega^\tau \langle \partial_t w(t), \rho(t) \rangle dt \\
+  \iint_{Q_{\omega,\tau}}  (-\Delta)^\frac12 w(x, t) \, (-\Delta)^{s-\frac12} \varphi(x, t)+ D_pH(x,Du) \cdot Dw \rho -L(x,D_pH(x,Du)) \eta \rho \, dxdt \\
=  \iint_{Q_{\omega,\tau}} f \eta \rho \,dxdt +  \iint_{Q_{\omega,\tau}} u \eta' \rho \,dxdt.
\end{multline}
Then, we use $w \in \H_2^1(Q_T)$ as a test function in the weak formulation of the adjoint problem satisfied by $\rho$ to get
\begin{equation}\label{inte1}
-\int_\omega^\tau \langle \partial_t\rho(t), w(t) \rangle dt + \iint_{Q_{\omega,\tau}}(-\Delta)^{s-\frac12}\rho(-\Delta)^\frac12 w  + D_pH(x,Du) \rho \cdot Dw \,dxdt = 0.
\end{equation}
We now fix $\omega$ small so that $\eta(\omega) = 0$. We then obtain, subtracting the previous equality to \eqref{loc1}, and integrating by parts in time, the identity
\begin{multline}\label{optloc}
\int_{\T}w(x,\tau)\rho_{\tau}(x)dx= \iint_{Q_{\omega,\tau}}\eta(t)f(x,t)\rho(x,t)dxdt\\
+\iint_{Q_{\omega,\tau}}\eta(t)L\big(x,D_pH(x,Du(x,t))\big)\rho(x,t)dxdt+\iint_{Q_{\omega,\tau}} \eta'(t)u(x,t)\rho(x,t)dxdt.
\end{multline}

For $h>0$ and $\xi\in\R^d$, $|\xi|=1$, define $\hat{\rho}(x,t):=\rho(x-h\xi,t)$. After a change of variables in the nonlocal problem \eqref{adj}, it can be seen that $\hat{\rho}$ satisfies, using $w$ as a test function,
\begin{multline}\label{inte2}
-\int_\omega^\tau \langle \partial_t \hat \rho(t), w(t) \rangle dt +\iint_{Q_{\omega,\tau}}(-\Delta)^{s-\frac12}\hat \rho(-\Delta)^\frac12 w\,dxdt\\
+ \iint_{Q_{\omega,\tau}}  D_pH(x-h\xi,Du(x-h\xi,t))\hat \rho(x,t) \cdot Dw(x,t) \,dxdt = 0.
\end{multline}
As before, we plug the vector field $\Theta(x,t) = D_pH(x-h\xi,Du(x-h\xi,t))$ and the test function $\varphi = \eta \hat \rho$ in \eqref{HJoptimalLOC} to conclude
\begin{multline*}
\int_\omega^\tau \langle \partial_t w(t), \hat \rho(t) \rangle dt + \\ \iint_{Q_{\omega,\tau}}(-\Delta)^{s-\frac12}\hat \rho(-\Delta)^\frac12 w + D_pH(x-h\xi,Du(x-h\xi,t)) \cdot Dw \hat \rho -L(x,D_pH(x-h\xi,Du(x-h\xi,t))) \eta \hat \rho \, dxdt \\
\le  \iint_{Q_{\omega,\tau}} f \eta \hat \rho \,dxdt +  \iint_{Q_{\omega,\tau}} u \eta' \hat \rho \,dxdt.
\end{multline*}
We subtract \eqref{inte2} to the previous inequality and obtain %after the change of variables $x \mapsto x + h\xi$ in the Lagrangian term
\begin{multline*}
\int_{\T}w(x,\tau)\hat{\rho}_{\tau}(x)dx \le 
%\iint_{Q_{s,\tau}} \partial_{j}\Big(\big(a_{ij}(x-h\xi,t) - a_{ij}(x,t)\big) \hat \rho(x,t)\Big)\partial_i w \,dxdt\\
\iint_{Q_{\omega,\tau}} L(x,D_pH(x-h\xi,Du(x-h\xi,t))) \eta \hat \rho \, dxdt+ \iint_{Q_{\omega,\tau}} f \eta \hat \rho \,dxdt \\+  \iint_{Q_{\omega,\tau}} u \eta' \hat \rho \,dxdt,
\end{multline*}
which, after the change of variables $x \mapsto x + h\xi$, becomes
\begin{multline}\label{suboptloc}
\int_{\T}w(x + h\xi,\tau){\rho}_{\tau}(x)dx \le 
%\iint_{Q_{s,\tau}} \partial_{j}\Big(\big(a_{ij}(x-h\xi,t) - a_{ij}(x,t)\big) \rho(x,t)\Big)\partial_i w \,dxdt\\
 \iint_{Q_{\omega,\tau}} \eta(t) L(x+h\xi,D_pH(x,Du(x,t)))\rho(x,t) \, dxdt + \iint_{Q_{\omega,\tau}} f \eta \hat \rho \,dxdt \\+  \iint_{Q_{\omega,\tau}} u \eta' \hat \rho \,dxdt\ .
\end{multline}
Taking the difference between \eqref{suboptloc} and \eqref{optloc} we conclude
\begin{multline}\label{step1}
\int_{\T}(w(x + h\xi, \tau)-  w(x,\tau)) {\rho}_{\tau}(x)dx \\
\le \iint_{Q_{\omega,\tau}} \eta(t) \Big(L(x+h\xi,D_pH(x,Du(x,t))) - L(x,D_pH(x,Du(x,t)))\Big)\rho(x,t) \, dxdt\\% \int_{\T}(w(x+h,0)-w(x,0)){\rho}(x,0)dx 
%\iint_{Q_{s,\tau}} \partial_{j}\Big(\big(a_{ij}(x-h\xi,t) - a_{ij}(x,t)\big) \rho(x,t)\Big)\partial_i w \,dxdt\\
%&\qquad + \iint_{Q_{s,\tau}} \eta(t) \Big(L(x+h\xi,D_pH(x,Du(x,t))) - L(x,D_pH(x,Du(x,t)))\Big)\rho(x,t) \, dxdt \\
+ \iint_{Q_{\omega,\tau}} \eta(t) f(x,t) \big(\rho(x-h\xi, t)- \rho(x,t) \big) \,dxdt \\
+  \iint_{Q_{\omega,\tau}} \eta'(t) u(x,t)\big( \rho(x-h\xi,t) - \rho(x,t) \big) \,dxdt =(I)+(II)+(III)
\end{multline}

\textbf{Step 2.} We now estimate all the right hand side terms of \eqref{step1}. We emphasize that constants $C, C_1, \ldots$ are not going to depend on $\tau, \rho_\tau, h, \xi$. \\
Proposition \ref{EstimateKfrac} and Corollary \ref{EstimateK1} applied with $\sigma=\frac{d+2s}{d+4s-1-\gamma'(2s-1)}$ yield
%First, we observe that in view of the representation formula given in \eqref{repr}, we have
%\begin{multline*}
%\iint_{Q_{\omega,\tau}}L(x,D_pH(x,Du(x,t))\rho(x,t)\,dxdt=\int_\T u(x,\tau)\rho_\tau(x)\,dx-\int_\T u(x,\omega)\rho(x,\omega)\,dx\\
%-\iint_{Q_{\omega,\tau}}f(x,t)\rho(x,t)\,dxdt\ .
%\end{multline*}
%Then, one obtains
%\[
%C_L^{-1}\iint_{Q_{\omega,\tau}}|D_pH(x,Du(x,t))|^{\gamma'}\rho(x,t)\,dxdt\leq C+\|f\|_{L^q(Q_\tau)}\|\rho\|_{L^{q'}(Q_{\omega,\tau})}\ .
%\]
%We use Corollary \ref{corembW} to get
%\begin{multline*}
%C_L^{-1}\iint_{Q_{\omega,\tau}}|D_pH(x,Du(x,t))|^{\gamma'}\rho(x,t)\,dxdt\\
%\leq C+C_1\|f\|_{L^q(Q_\tau)}\left(\iint_{Q_{\omega,\tau}}|D_pH(x,Du)|^{\frac{d+2s}{(2s-1)q}}\rho(x,t)\,dxdt+\|\rho_\tau\|_{L^{1}(\T)}\right)\ .
%\end{multline*}
%A straightforward application of Young's inequality yields a control on $\iint_{Q_{\omega,\tau}}|D_pH(x,Du(x,t))|^{\gamma'}\rho(x,t)\,dxdt$ and hence we derive by Theorem \ref{estFP2frac}
\[
\iint_{Q_{\omega,\tau}}|D_pH(x,Du(x,t))|^{\gamma'}\rho(x,t)\,dxdt+\|\rho\|_{\H_{\frac{d+2s}{d+4s-1-\gamma'(2s-1)}}^{2s-1}(Q_{\omega,\tau})}\leq C_2\ .
\]
Let $\alpha\in(0,1)$ to be determined later. For $\nu=D_pH(x,p)$ we have $L(x,\nu)=\nu\cdot p-H(x,p)$ and thus
\[
L(x+h\xi,D_pH(x,Du(x,t))) - L(x,D_pH(x,Du(x,t))\leq H(x,Du(x,t))-H(x+\xi,Du(x,t))\ .
\]
Next, using ($H_\alpha$) and the above inequality, we get
\begin{multline*}
\left| \iint_{Q_{\omega,\tau}} \eta(t) \Big(L(x+h\xi,D_pH(x,Du(x,t))) - L(x,D_pH(x,Du(x,t)))\Big)\rho(x,t) \, dxdt \right|  \\
%\le C_1|h|^\alpha \iint_{Q_{\omega,\tau}}  [ L\big(\cdot, D_pH(x,Du(x,t))\big)]_{B^{\alpha}_{\infty\infty}(\T)} \rho(x,t) \, dxdt  \\
%\leq C_2|h|^\alpha \iint_{Q_{\omega,\tau}}  |D_xL\big(x, D_pH(x,Du(x,t))| \rho(x,t) \, dxdt  \\
\le C_2C_L |h|^\alpha \iint_{Q_{\omega,\tau}} \big(|D_pH(x,Du(x,t))|^{\gamma'} +1\big) \rho(x,t) \, dxdt \le C |h|^\alpha.
\end{multline*}
We can apply the Sobolev embedding in Lemma \ref{embs}-(ii) with $\delta=q'$, $p=\frac{d+2s}{d+4s-1-\gamma'(2s-1)}(\leq\frac{d+2s}{d+2s-1}<2)$, giving $\alpha=\gamma'(2s-1)-\frac{d+2s}{q}$, and Lemma \ref{embnik}, to show
\begin{multline*}
\left|  \iint_{Q_{\omega,\tau}} \eta(t) f(x,t) \big(\rho(x-h\xi, t)- \rho(x,t) \big) \,dxdt \right |   \\
\le |h|^\alpha \iint_{Q_{\omega,\tau}} |f(x,t)|\,\frac{|\big(\rho(x-h\xi, t)- \rho(x,t) \big)|}{|h|^\alpha} \,dxdt \le |h|^\alpha \|f\|_{L^{q}(Q_{\omega,\tau})} \|\rho\|_{L^{q'}(\omega,\tau;N^{\alpha,q'}(\T))}\\
\leq  C_1|h|^\alpha \|f\|_{L^{q}(Q_{\omega,\tau})} \|\rho\|_{L^{q'}(\omega,\tau;H^{\alpha}_{q'}(\T))}
\leq  C_2|h|^\alpha \|f\|_{L^{q}(Q_{\omega,\tau})}\|\rho\|_{\H_{\frac{d+2s}{d+4s-1-\gamma'(2s-1)}}^{2s-1}(Q_{\omega,\tau})}\\
\leq C_3|h|^\alpha \|f\|_{L^{q}(Q_{\omega,\tau})}\ .
\end{multline*}
Finally, as above, we conclude
\begin{multline*}
\left|  \iint_{Q_{\omega,\tau}} \eta'(t) u(x,t)\big( \rho(x-h\xi,t) - \rho(x,t) \big) \,dxdt \right| \\
\le |h|^\alpha \big(\sup_{(0, T)}|\eta'(t)| \big) \|u\|_{L^{q'}(Q_{\omega,\tau})} \|\rho\|_{L^{q'}(\omega,\tau;N^{\alpha,q'}(\T))} \\
\leq C_1|h|^\alpha \big(\sup_{(0, T)}|\eta'(t)| \big) \|u\|_{C(\overline{Q}_{T})}\|\rho\|_{L^{q'}(\omega,\tau;H^{\alpha}_{q'}(\T))}
\le C_2|h|^\alpha \sup_{(0, T)}|\eta'(t)|.
\end{multline*}
Plugging all the estimates in \eqref{step1} we obtain
\begin{equation}\label{step2}
  \int_{\T}(w(x + h\xi, \tau)-  w(x,\tau)) {\rho}_{\tau}(x)dx   \\ \le C|h|^\alpha \Big(%\eta(0) \|D u_0\|_{L^\infty(Q_T)} + 
  \sup_{(0, T)}|\eta'(t)| + 1 \Big)
\end{equation}
when $q>\frac{d+2s}{(2s-1)\gamma'}$.\\
\par\smallskip
\textbf{Step 3.} Since \eqref{step2} holds for all smooth $\rho_\tau \ge 0$ with $\|\rho_\tau\|_{L^1(\T)} = 1$, we get
\begin{equation*}
\eta(\tau) [u(x + h\xi, \tau)-  u(x,\tau)] \le C|h|^\alpha \Big(%\eta(0) \|D u_0\|_{L^\infty(Q_T)} + 
\sup_{(0, T)}|\eta'(t)| + 1 \Big)
\end{equation*}
for all $x \in \T$, $\xi \in \R^d$, $h > 0$. Thus, $u(\cdot, \tau)$ is H\"older continuous, and
\begin{equation*}
\eta(\tau) [u(\cdot, \tau)]_{C^\alpha(\T)} \le C \Big(%\eta(0) \|D u_0\|_{L^\infty(Q_T)} + 
\sup_{(0, T)}|\eta'(t)| + 1 \Big).
\end{equation*}
Since $C$ does not depend on $\tau \in (0,T)$, we take $t_1\in(0,T)$, $\eta=\eta(t)$ nonnegative smooth function on $[0,T]$ such that $\eta(t)\leq1$, $\eta(t)=1$ on $[t_1,T]$ and vanishing on $[0,t_1/2]$. This proves Theorem \ref{mainholder}-(C).\\

%
%Finally, for the special case $D a \equiv 0$ on $Q_T$, one may follow the very same lines, with the difference that there is no need to control the term appearing in \eqref{aterm} (which is identically zero). Therefore, there is no need to keep track of $\|\eta Du\|_{L^{(d+2)(\gamma-1)}(Q_\tau)}$, and therefore the theorem is proven without assuming the constraint $\sP = \sQ$ in \eqref{conditionAS}.
%To prove (ii), we have to modify the estimate in Step 2 of terms (II) and (III). In fact, it is immediate to check that in view of the integrability results for the adjoint variable in Corollary \ref{EstimateK1} one needs again $f\in L^r(Q_{s,\tau})$ for $r>\frac{d+2s}{(2s-1)\gamma'}$. Precisely, in Step 2, the treatment of (II) reads as follows. Here, we note that for
%\[
%\alpha=2s-1+\frac{d+2s}{r'}-\frac{d+2s}{q'}
%\]
%and $q\geq\frac{d+2s}{(2s-1)(\gamma'-1)}$ we have
%\[
%\alpha\leq (2s-1)\gamma'-\frac{d+2s}{r}\ .
%\]
%We can apply the Sobolev embedding in Lemma \ref{embs} to show that
%\begin{multline*}
%\left|  \iint_{Q_{\omega,\tau}} \eta(t) f(x,t) \big(\rho(x-h\xi, t)- \rho(x,t) \big) \,dxdt \right |   \\
%\le |h|^\alpha \iint_{Q_{\omega,\tau}} |f(x,t)|\,\frac{|\big(\rho(x-h\xi, t)- \rho(x,t) \big)|}{|h|^\alpha} \,dxdt \le C_1 |h|^\alpha \|f\|_{L^{r}(Q_{\omega,\tau})} \|\rho\|_{L^{r'}(\omega,\tau;N^{\alpha,r'}(\T))}\\
%\leq  C_2 |h|^\alpha \|f\|_{L^{r}(Q_{\omega,\tau})}\|\rho\|_{\H_{q'}^{2s-1}(Q_{\omega,\tau})}\ .
%\end{multline*}
To prove the global-in-time bound in (D) one may observe that if $u\in \H_q^{2s}(Q_T)$ is a strong solution with $q>\frac{d+2s}{(2s-1)\gamma'}$, $\gamma\geq 2s$ (or a global weak solution), then the solution is global in time and one can take $\omega=0$ throughout the proof setting also $\eta\equiv1$ on $[0,T]$. Being the solution global, norms $\|u\|_{C(\overline{Q}_T)}$ can be replaced by $\|u_0\|_{C(\T)}$ by Theorem \ref{supnorm}, which are in turn bounded by $\|u_0\|_{C^\alpha(\T)}$. Now, an additional term of the form
\[
\int_\T \frac{u(x+h,0)-u(x,0)}{|h|^\alpha}\rho(x,0)\,dx
\]
arises, which can be immediately bounded by $[u_0]_{C^\alpha(\T)}$ since $\int_\T \rho(0)=1$.

\end{proof}

\begin{rem}
Using the same scheme of Theorem \ref{mainholder}, we believe one can even handle the case $\gamma<2s$ by considering appropriate weak solutions (not continuous on the whole cylinder $Q_T$) to \eqref{hjb}. 
\end{rem}

\begin{rem}\label{lipnonloc}
An approach similar to that for the H\"older bounds and the one in \cite[Theorem 1.1]{CG2}, which exploits the regularity properties in Corollary \ref{EstimateK1}, yields a Lipschitz regularization effect for \eqref{hjb} whenever $f\in L^q(0,T;H_q^{2-2s}(\T))$ for $\sigma=q$ as in Corollary \ref{EstimateK1}. This requires to impose that $|D_{x}H(x,p)|\leq C_H(|p|^{\gamma}+1)$ instead of ($H_\alpha$), see \cite[Chapter 7]{TesiAle}.
\end{rem}

\subsection{Maximal $L^q$-regularity}
\subsubsection{An overview of the results in the viscous case}\label{sec;conj}
Let us first consider the following viscous problem
\begin{equation}\label{intro}
-\Delta u+|Du|^\gamma=f(x)\text{ in }\Omega\ ,\gamma>1\ ,
\end{equation}
where $f$ is an unbounded source term belonging to a suitable Lebesgue space $L^q$. In \cite{LionsSeminar,Napoli} P.-L. Lions proposed the following conjecture:
\begin{conj}\label{cl}
Let $f\in L^q(\Omega)$, $q>1$, for some
\begin{equation}\label{reg1}
q>\frac{d}{\gamma'}=\frac{d(\gamma-1)}{\gamma}\ .
\end{equation}
and $\gamma>1$. Then, every solution to \eqref{intro} satisfies the a priori estimate
\[
\|D^2u\|_{L^q}+\||Du|^\gamma\|_{L^q}\leq C(\|f\|_{L^q},d,q,\gamma)\ .
\]
Moreover, the estimate is false when $q\leq d/\gamma'$.
\end{conj}
A byproduct of this statement is a maximal $L^q$-regularity for solutions to \eqref{intro}. In other words, \eqref{intro} behaves, in terms of regularity, as the Poisson equation (cf \cite[Theorem 9.9]{GT}) under the regime \eqref{reg1} of the integrability exponent $q$ of the right-hand side $f$. 
%Maximal regularity at the endpoints $q=1$ and $q=\infty$ fails to be true even for the Poisson equation, see \cite{Pigolasurvey}, and one needs to consider e.g. Besov scales \cite{OS}. 
Conjecture \ref{cl} completes the results in \cite{Lions85} for the subcritical range of the integrability of the forcing term, where it is shown a Lipschitz regularity result when $f\in L^q$, $q>d$, and every $\gamma>1$, obtained via an integral Bernstein method. A proof of Conjecture \ref{cl} has been proposed in \cite{CG4} appropriately modifying the Bernstein method, while an extension to the parabolic viscous framework has been already provided in \cite{CG5}. More precisely, in \cite{CG5} it is proved that maximal regularity for viscous \eqref{hjb} occurs for strong solutions when $f\in L^q(Q_T)$ with 
\[
q > 
\begin{cases}
(d+2)\frac{\gamma-1}{\gamma} & \text{if $1 + \frac{2}{d+2} < \gamma < 2$} \medskip \\ 
(d+2)\frac{\gamma-1}{2} & \text{if $\gamma \ge 2$}.
\end{cases}
\]
Note that the threshold $q=(d+2)(\gamma-1)/\gamma$ can be regarded as a parabolic analogue to the one in \eqref{reg1}. We also refer to \cite{CPorr} for more recent maximal regularity results for viscous ($s=1$) problems with quadratic gradient growth and right-hand sides in mixed Lebesgue scales, and to \cite{Swiech} (and the references therein) for some maximal regularity properties for fully nonlinear  second order uniformly parabolic problems in the context of $L^p$ viscosity solutions.

\subsubsection{The fractional case}

We first recall the following Calder\'on-Zygmund regularity result for the fractional heat equation with unbounded potential.
\begin{lemma}\label{maximal}
Let $u\in\H_q^{2s}(Q_T)$, $q>1$, be a strong solution to 
\[
\begin{cases}
\partial_t u+(-\Delta)^su=V(x,t)&\text{ in }Q_T\\
u(x,0)=u_0&\text{ in }\T\ .
\end{cases}
\]
with $V\in L^q(Q_T)$ and $u_0\in (L^q(\T),H_q^{2s}(\T))_{1-1/q,q}\simeq W^{2s-2s/q,q}(\T)$. Then, there exists a constant $C$ that remains bounded for bounded values of $T$, such that
\[
\|u\|_{\H_q^{2s}(Q_T)}=\|\partial_t u\|_{L^q(Q_T)}+\| u\|_{L^q(0,T;H_q^{2s}(\T))}\leq C(\|V\|_{L^q(Q_T)}+\|u_0\|_{W^{2s-2s/q,q}(\T)})\ .
\]
As a consequence, every strong solution to \eqref{hjb} with $u_0\in W^{2s-2s/q,q}(\T)$ satisfies
\[
\|u\|_{\H_q^{2s}(Q_T)}\leq C(\|f\|_{L^q(Q_T)}+\|H(x,Du)\|_{L^q(Q_T)}+\|u_0\|_{W^{2s-2s/q,q}(\T)})\ .
\]
\end{lemma}
By gathering all the previous results and the estimates in Theorems \ref{supnorm} and Theorem \ref{integest} we have the following maximal $L^q$-regularity result for \eqref{hjb} with (fractional) sub-natural growth.
\begin{proof}[Proof of Theorem \ref{maxregmain}]
We exploit the Gagliardo-Nirenberg inequality in Lemma \ref{GN1} to get for $\gamma\in(1,2s)$
\begin{equation}\label{GNbo}
\||Du(t)\|_{L^{\gamma q}(Q_T)}\leq C_1\|u(t)\|_{H_q^{2s}(\T)}^\theta\|u(t)\|_{L^{z}(\T)}^{1-\theta}
\end{equation}
for $z\in(1,\infty)$ and $\theta\in\left[\frac{1}{2s},1\right)$ with
\[
\frac{1}{\gamma q}=\frac{1}{d}+\theta\left(\frac1q-\frac{2s}{d}\right)+\frac{1-\theta}{z}\ .
\]
By Theorems \ref{supnorm} and \ref{integest} we have
\[
\|u(t)\|_{L^\infty(0,T;L^z(\T))}<\infty
\]
for any $z\leq p=\frac{dq}{d+2s-2sq}$ if $q<\frac{d+2s}{2s}$, $z\in[1,\infty)$ when $q=\frac{d+2s}{2s}$, $z=\infty$ for $q>\frac{d+2s}{2s}$. Since $q>\frac{d+2s}{(2s-1)\gamma'}$, we conclude $p>\frac{d(\gamma-1)}{2s-\gamma}$. Then, we choose $z$ close to $\frac{d(\gamma-1)}{2s-\gamma}$ so that 
\[
\theta\gamma=\frac{\frac1z+\frac1d-\frac{1}{\gamma q}}{\frac1z+\frac{2s}{d}-\frac1q}\gamma<1
\]
and $\theta\in[1/2s,1/\gamma)$. Raising \eqref{GNbo} to $\gamma q$ and integrating in time we have
\[
\|Du\|_{L^{\gamma q}(Q_T}\leq C_2\|u\|_{L^q(0,T;H_q^{2s}(\T))}^{\gamma\theta}\|u\|_{L^\infty(0,T;L^{z}(\T))}^{\gamma(1-\theta)}\ .
\]
Then, by \eqref{H} we deduce for positive constants $C_3,C_4>0$
\[
\|H(x,Du)\|_{L^q(Q_T)}\leq C_3(1+\||Du|\|^\gamma_{L^{\gamma q}(Q_T)})\leq C_4(\|u\|_{L^q(0,T;H_q^{2s}(\T))}^{\gamma\theta}\|u\|_{L^\infty(0,T;L^{z}(\T))}^{\gamma(1-\theta)}+1)\ .
\]
Using Lemma 	\ref{maximal} and Young's inequality we have
\begin{multline*}
\|u\|_{\H_q^{2s}(Q_T)}\leq C_5(\|H(x,Du)\|_{L^q(Q_T)}+\|f\|_{L^q(Q_T)}+\|u_0\|_{W^{2-\frac{2}{q},q}(\T)})\\
\leq C_6(\|u\|_{L^q(0,T;H_q^{2s}(\T))}^{\gamma\theta}\|u\|_{L^\infty(0,T;L^{z}(\T))}^{\gamma(1-\theta)}+\|f\|_{L^q(Q_T)}+\|u_0\|_{W^{2s-\frac{2s}{q},q}(\T)})\\
\leq \frac12\|u\|_{L^q(0,T;H_q^{2s}(\T))}+C_7\|u\|_{L^\infty(0,T;L^{z}(\T))}^{\frac{\gamma(1-\theta)}{1-\gamma\theta}}+C_6(\|f\|_{L^q(Q_T)}+\|u_0\|_{W^{2s-\frac{2s}{q},q}(\T)})\ .
\end{multline*}
We then absorb the term $\frac12\|u\|_{\H_q^{2s}(Q_T)}$ on the left-hand side and use the integral estimate in Theorem \ref{integest} in $L^\infty(L^z)$ to conclude the assertion. Note that here we used Remark \ref{embin} to make sure that $\|u_0\|_{L^p(\T)}$ is bounded by $\|u_0\|_{W^{2s-2s/q,q}(\T)}$.  Then, we have
\[
\||Du|\|_{L^{\gamma q}(Q_T)}\leq C(\|f\|_{L^q(Q_T)},\|u_0\|_{W^{2s-\frac{2s}{q},q}(\T)})
\]
for a possibly different constant $C>0$. To prove the case $\gamma\geq 2s$, we begin with the H\"older bound in Theorem \ref{mainholder}
\[
\sup_{t\in[0,T]}\|u(t)\|_{C^\alpha(\T)}<\infty
\]
valid for $\gamma\geq 2s$ as follows:
\begin{itemize}
\item For $\gamma=2s$ we get $\alpha= 2s-\frac{d+2s}{q}$ such that $\alpha\in(0,2s-1)$ when $q<d+2s$, $\alpha<1$ when $d+2s\leq q<\frac{d+2s}{2s-1}$, while we have any $\alpha\in(0,1)$ when $q\geq\frac{d+2s}{2s-1}$;
\item For $2s\leq \gamma<\frac{1}{2-2s}$ (i.e. $\gamma'(2s-1)>1$), we get $\alpha= \gamma'(2s-1)-\frac{d+2s}{q}$ such that $\alpha\in(0,2s-1)$ if $q<\frac{d+2s}{(2s-1)(\gamma'-1)}$, $\alpha\in(2s-1,1)$ if $\frac{d+2s}{(2s-1)(\gamma'-1)}\leq q<\frac{(d+2s)(\gamma-1)}{1-(2-2s)\gamma}$, while we have any $\alpha\in (0,1)$ if $q\geq \frac{(d+2s)(\gamma-1)}{1-(2-2s)\gamma}$;
\item For $\gamma\geq\frac{1}{2-2s}$ (i.e. $\gamma'(2s-1)\leq1$), we get $\alpha=\gamma'(2s-1)-\frac{d+2s}{q}\in (0,1)$ when $q>\frac{d+2s}{\gamma'(2s-1)}(\geq d+2s)$.
\end{itemize}
We use Lemma \ref{GN2} to get for a.e. $t\in(0,T)$
\[
\|Du(t)\|_{L^{\gamma q}(\T)}\leq C\|u(t)\|_{H_q^{2s}(\T)}^{\beta}\|u(t)\|_{C^{\alpha}(\T)}^{1-\beta}
\]
and hence after raising the above inequality to the power $\gamma q$ and integrating in time we get
\[
\|Du\|_{L^{\gamma q}(Q_T)}\leq C\|u\|_{L^q(0,T;H_q^{2s}(\T))}^{\beta}\|u\|_{L^\infty(0,T;C^{\alpha}(\T))}^{1-\beta}\ ,
\]
where
\[
\beta\in\left[\frac{1-\alpha}{2s-\alpha},1\right)\ ,\alpha\neq 2s-\frac{d}{q}\ .
\]
Choosing $\beta=\frac{1-\alpha}{2s-\alpha}$ (or $\alpha$ close to 1 when $\gamma<\frac{1}{2-2s}$ and $q\geq \frac{(d+2s)(\gamma-1)}{1-(2-2s)\gamma}$) we get $\beta\gamma<1$ when 
\[
\alpha=\gamma'(2s-1)-\frac{d+2s}{q}>\frac{\gamma-2s}{\gamma-1}\ ,
\]
i.e. when $\gamma<s/(1-s)$
\[
q>\frac{(d+2s)(\gamma-1)}{2s-(2-2s)\gamma}\ .
\]
Observe also that $\frac{1}{2-2s}<\frac{2s}{2-2s}=\frac{s}{1-s}$. In particular, when $\gamma=2s$ we have $q>\frac{(d+2s)(2s-1)}{4s^2-2s}=\frac{d+2s}{2s}$. 
Then, we deduce 
\[
\|H(x,Du)\|_{L^q(Q_T)}\leq C_1(1+\||Du|\|^{\gamma}_{L^{\gamma q}(Q_T)})\leq C_2(\|u\|_{L^q(0,T;H_q^{2s}(\T))}^{\gamma\beta}\|u\|_{L^\infty(0,T;C^\alpha(\T))}^{\gamma(1-\beta)}+1)\ ,
\]
which allows to exploit generalized Young's inequality since $\gamma\beta<1$ and conclude as in the case $\gamma<2s$.
\end{proof}

\begin{rem}\label{restr}
The restriction $\gamma<\frac{s}{1-s}$ is not new in the literature. Under the assumption $\gamma\in\left[\frac{d+2s}{d+1},\frac{s}{1-s}\right)$ it has been recently proved existence of solutions for fractional Kardar-Parisi-Zhang equations with right-hand side in $L^q$ under suitable restrictions on $q$ for problems posed on bounded $C^{1,1}$ domains of $\R^d$. 
\end{rem}

\appendix
\section{Nonlocal equations with drift terms}\label{eqdrift}
We begin with the following comparison principle for \eqref{drift}, whose proof follows the strategy implemented in \cite{Feo}. By weak solutions to \eqref{drift} we mean a function in the space $\H_2^1(Q_T)$ satisfying the equation in the sense of distributions.
\begin{prop}\label{comp}
Let $v_1,v_2\in L^2(\omega,\tau;H^1_2(\T))$ be respectively a weak sub- and supersolution to \eqref{drift} with $b\in L^{\sP}(Q_{\omega,\tau})$, $\sP\geq \frac{d+2s}{2s-1}$ such that $v_1(\omega)\leq v_2(\omega)$. Then $v_1\leq v_2$ in $Q_{\omega,\tau}$.
\end{prop}
\begin{proof}
Let us suppose that $v_1,v_2$ are two weak sub- and supersolutions to \eqref{drift} such that $w=v_1-v_2>0$ in a subset $D\subset Q_T$ such that $|D|>0$. Denote by 
\[
w_k=\begin{cases}
w^+-k&\text{ if }w^+>k\\
0&\text{ otherwise }
\end{cases}
\]
for $k\in(0,M)$, $M=\sup_D w^+$. We use $\varphi=(-\Delta)^{1-s}w_k\in L^2(H^{2s-1})$ as a test function in the weak formulation of the difference of the two equations. We then get
\[
\mathrm{esssup}_{t\in(0,T)}\int_\T |(-\Delta)^{\frac{1-s}{2}}w_k|^2\,dx+\iint_{Q_T}|(-\Delta)^\frac12 w_k|^2\,dxdt\leq \iint_{Q_T}|b||Dw_k||(-\Delta)^{1-s}w_k|\,dxdt\ .
\]
Denote by 
\[
E_k:=\{(x,t)\in Q_T:k<w^+<\sup_Dw^+\}\ .
\]
We then note that by H\"older's inequality
\begin{multline*}
\iint_{Q_T}|b||Dw_k||(-\Delta)^{1-s}w_k|\,dxdt\leq \|b\|_{L^{\frac{d+2s}{2s-1}}(E_k)}\|Dw_k\|_{L^2(Q_T)}\|(-\Delta)^{1-s}w_k\|_{L^{\frac{2(d+2s)}{d+2-2s}}(Q_T)}\\
\leq C\|b\|_{L^{\frac{d+2s}{2s-1}}(E_k)}\|Dw_k\|_{L^2(Q_T)}\|w_k\|_{L^\frac{2(d+2s)}{d+2-2s}(0,T;H^{2-2s}_{\frac{2(d+2s)}{d+2-2s}}(\T))}\ .
\end{multline*}
We use the fractional Gagliardo-Nirenberg inequality on Besov scales \cite[Theorem 4.1]{HMOW} saying that
\[
\|u\|_{B^{\nu}_{p1}(\T)}\leq C\|u\|_{B^{\nu_0}_{p_0\infty}(\T)}^{1-\theta}\|u\|_{B^{\nu_1}_{p_1\infty}(\T)}^\theta
\]
for
\[
\frac{d}{p}-\nu=(1-\theta)\left(\frac{d}{p_0}-\nu_0\right)+\theta \left(\frac{d}{p_1}-\nu_1\right);
\]
\[
\nu_0-\frac{d}{p_0}\neq \nu_1-\frac{d}{p_1};
\]
\[
\nu\leq (1-\theta)\nu_0+\theta \nu_1\ .
\]
We then take $\nu=2-2s$, $q=1$, $p_0=p_1=2$, $\nu_0=1-s$, $\nu_1=1$ and get using Lemma \ref{inclBesovBessel}-(iii)
\[
\|u(t)\|_{H^{2-2s}_p(\T)}\leq C\|u(t)\|_{H^{1-s}_2(\T)}^{1-\theta}\|u(t)\|_{H^{1}_2(\T)}^\theta
\]
with $\theta\in(1/s-1,1)$ and $p$ fulfilling the above compatibility condition. We then integrate in time to get
\[
\int_0^T\|u(t)\|_{H^{2-2s}_p(\T)}^p\,dt\leq C \int_0^T \|u(t)\|_{H^{1-s}_2(\T)}^{(1-\theta)p}\|u(t)\|_{H^{1}_2(\T)}^{\theta p}\,dt\ .
\]

We take $\theta=2/p$ for $p=\frac{2(d+2s)}{d+2-2s}$, so that $\theta=\frac{d+2-2s}{d+2s}$ (hence $\theta+1=\frac{2d+2}{d+2s}>\frac1s$), $\theta p=2$ and $(1-\theta)p=\frac{4(2s-1)}{d+2-2s}$. Then, it is immediate to check that 
\[
\frac{d}{p}-2+2s=(1-\theta)\left(\frac{d}{2}-1+s\right)+\theta \left(\frac{d}{2}-1\right) .
\]
This yields
\begin{equation}\label{embmix}
\|u\|_{L^\frac{2(d+2s)}{d+2-2s}(0,T;H^{2-2s}_{\frac{2(d+2s)}{d+2-2s}}(\T))}^{2\frac{d+2s}{d+2-2s}}\leq C\|u\|_{L^\infty(0,T;H^{1-s}(\T))}^{\frac{4(2s-1)}{d+2-2s}}\left(\int_0^T\|u(t)\|_{H^1_2(\T)}^2\right)\ .
\end{equation}
Note that this agrees with the classical viscous case $s=1$, cf \cite[Lemma 3.1]{Feo}. This yields the embedding
\[
L^\infty(0,T;H^{1-s}_2(\T))\cap L^2(0,T;H^1(\T)) \hookrightarrow L^\frac{2(d+2s)}{d+2-2s}(0,T;H^{2-2s}_{\frac{2(d+2s)}{d+2-2s}}(\T))\ .
\]
Embedding \eqref{embmix}, after applying Young's inequality, gives
\begin{multline*}
\|u\|_{L^\frac{2(d+2s)}{d+2-2s}(0,T;H^{2-2s}_{\frac{2(d+2s)}{d+2-2s}}(\T))}\leq C_1\|u\|_{L^\infty(0,T;H^{1-s}(\T))}^{\frac{2(2s-1)}{d+2s}}\|u\|_{L^2(0,T;H^1_2(\T))}^\frac{d+2-2s}{d+2s}\\
\leq C_2(\|u\|_{L^\infty(0,T;H^{1-s}(\T))}+\|u\|_{L^2(0,T;H^1(\T))})\ .
\end{multline*}
We then conclude
\begin{multline}\label{contr}
\left(\mathrm{esssup}_{t\in(0,T)}\|(-\Delta)^{\frac{1-s}{2}}w_k\|_{L^2(\T)}+\|(-\Delta)^\frac12w_k\|_{L^2(Q_T)}\right)^2\\
\leq C\|b\|_{L^{\frac{d+2s}{2s-1}}(E_k)}\left(\mathrm{esssup}_{t\in(0,T)}\|(-\Delta)^{\frac{1-s}{2}}w_k\|_{L^2(\T)}+\|(-\Delta)^\frac12w_k\|_{L^2(Q_T)}\right)^2\ .
\end{multline}
As $k\to M$ this yields $\|b\|_{L^{\frac{d+2s}{2s-1}}(E_k)}\to0$, and hence \eqref{contr} gives a contradiction.
\end{proof}

We have the following existence result obtained by similar arguments used for elliptic problems, cf \cite[Lemma 3.4]{AP}, \cite{Feo} and the references therein.
\begin{prop}
Let $v_\omega\in H^{1-s}(\T)$ and $b\in  L^\sP(Q_{\omega,\tau})$, $\sP\geq\frac{d+2s}{2s-1}$. Then, there exists a solution $v\in\H^{1}_2(Q_{(\omega,\tau)})$ to \eqref{drift}.
\end{prop}

\begin{proof}
We prove the existence via the Leray-Schauder fixed point theorem. Define the map $\Psi:\H^{1}_2(Q_{\omega,\tau})\times[0,1]\to \H^{1}_2(Q_{\omega,\tau})$ as the map $z\longmapsto \Psi[z;\sigma]=v$ with $v$ solving the parametrized equation
\[
\partial_tv+(-\Delta)^sv=\sigma b(x,t)\cdot Dz\text{ in }Q_\omega\ , v(x,\omega)=\sigma v_\omega\text{ in }\T\ ,
\]
for the parameter $\sigma\in[0,1]$. First, observe that $\Psi[z;0]=0$ by standard results for fractional heat equations. We then observe that $\H_2^1\hookrightarrow L^{\frac{2(d+2s)}{d+2s-2}}\hookrightarrow L^{\frac{2(d+2s)}{d+2-2s}}$. By H\"older's inequality we observe $b\cdot Dz\in L^2_{x,t}$ when $b\in L^\sP_{x,t}$ with $\sP\geq\frac{d+2s}{2s-1}$ since $Dz\in L^{\frac{2(d+2s)}{d+2-2s}}$. Then, maximal regularity yields $v\in \H_2^{2s}$ and hence $v\in L^2(H^1)$. Finally, by using the equation we get $\partial_t v\in L^2(H^{-1})$, so that the map is well-defined in $\H_2^1$.
%We also note that $b\cdot Dz\in \mathbb{H}_2^{1-2s}(Q_{\omega,\tau})$ by a straightforward consequence of Sobolev inequality in Lemma \ref{inclstat}-(iii) (with $\mu=2s-1$) and H\"older's inequality. Indeed, we have
%\begin{multline*}
%\||b||Dz|\|_{L^2(\omega,\tau;H^{1-2s}(\T))}\leq \||b||Dz|\|_{L^2(\omega,\tau;L^{\frac{2d}{d+2(2s-1)}}(\T))}\\
%\leq C \||b|\|_{L^\infty(\omega,\tau;L^{\frac{d}{2s-1}}(\T))}\|Dz\|_{L^2(Q_\omega)}\ ,
%\end{multline*}
%which is bounded when $\sQ=\infty$, $\sP=\frac{d}{2s-1}$.
%\begin{multline*}
%\||b||Dz|\|_{L^2(\omega,\tau;H^{1-2s}(\T))}\leq \||b||Dz|\|_{L^2(\omega,\tau;L^{\frac{2d}{d+2(2s-1)}}(\T))}\\
%\leq C\||b||Dz|\|_{L^\frac{2(d+2s)}{d+4s-2}(Q_{\omega,\tau})}\leq C\|b\|_{L^{\frac{d+2s}{2s-1}}(Q_{\omega,\tau})}\|Dz\|_{L^2(Q_{\omega,\tau})}
%\end{multline*}
%Then, in view of Theorem \ref{fracregtor}, we can infer that $v\in\H_2^{1}(Q_{\omega,\tau})$.
%%\[
%%\|v\|_{\H_2^{1}(Q_\omega)}\leq C(\||b||Dz|\|_{\mathbb{H}_2^{1-2s}(Q_{\omega,\tau})}+\|v_\omega\|_{H^{1-s}(\T)})
%%\]
%since $\sigma\in[0,1]$. This shows that the map $\Psi$ is well-defined. 
We now prove the compactness of the map. Let $z_n$ be a bounded sequence in $\H^{1}_2(Q_{\omega,\tau})$ and $v_n=\Psi[z_n,\sigma]$. Arguing as above, we get $v_n\in\H_2^1(Q_{\omega,\tau})$ and then exploit the compactness of $\H^{1}_2(Q_{\omega,\tau})$ onto $L^2(Q_{\omega,\tau})$ (see Lemma \ref{compact} with $\mu=1$), so as to have the strong convergence of $v_n$ to $v$ in $L^2(Q_{\omega,\tau})$ and the weak convergence of $(-\Delta)^{\frac12}v_n$ to $(-\Delta)^{\frac12}v$ in $L^2(Q_\omega)$ along a subsequence. The compactness of $\Psi$ follows by using now the test function $\varphi:=(-\Delta)^{1-s}(v_n-v)$ that satisfies the requirements $\varphi\in L^2(H^{2s-1})$ with $\partial_t\varphi\in L^2(H^{-1})$. 
%We first observe that by duality $\partial_t v_n\in L^2(\omega,\tau;H^{-1}(\T))$ since for $\psi\in L^2(H^{1})$
%\[
%\iint_{Q_{\omega,\tau}}\partial_tv_n\psi\leq C\|D\psi\|_{L^2(Q_{(\omega,\tau)})}(\|Dv_n\|_{L^2(Q_{(\omega,\tau)})}+\||b||Dz_n|\|_{\mathbb{H}_2^{1-2s}(Q_\omega)})\ .
%\]
We then write 
\begin{multline*}
\iint_{Q_{\omega,\tau}}|(-\Delta)^\frac12(v_n-v)|^2\,dxdt 
\leq\iint_{Q_{\omega,\tau}}b\cdot Dz_n (-\Delta)^{1-s}(v_n-v)\,dxdt-\iint_{Q_\omega}(-\Delta)^{\frac12}v\cdot (-\Delta)^{\frac12}(v_n-v)\,dxdt\\
-\iint_{Q_{\omega,\tau}}\partial_tv_n(-\Delta)^{1-s}(v_n-v)\,dxdt\ ,
\end{multline*}
which shows the strong convergence of $(-\Delta)^{\frac12}v_n$ to $(-\Delta)^{\frac12}v$ in $L^2(Q_{\omega,\tau})$ by using that $\partial_tv_n\in L^2(H^{-1})$, the weak convergence of $(-\Delta)^{\frac12}v_n$ to $(-\Delta)^{\frac12}v$ in $L^2(Q_\omega)$ and the fact that $|b||Dz|\in L^2(Q_{\omega,\tau})$ (note that we have also that $(-\Delta)^{1-s} v_n$ converges weakly to $(-\Delta)^{1-s} v$). The strong convergence of the time derivative follows then by duality.\\
To prove the a priori estimate for every fixed point $g\in \H_2^1(Q_{\omega,\tau})$ of the map $\Psi$, that is satisfying $g= \Psi[g;\sigma]$, we argue by contradiction as in the elliptic case. Indeed, suppose that for any $n$ one has $v_n\in  \H_2^1(Q_{\omega,\tau})$, $\sigma_n\in[0,1]$ such that $v_n=\Psi[v_n;\sigma_n]$ with $\|v_n\|_{ \H_2^1(Q_{\omega,\tau})}\geq n$. This means that
\[
\|v_n\|_{ \H_2^1(Q_{\omega,\tau})}\to\infty\ .
\]
The condition $v_n=\Psi[v_n;\sigma_n]$ gives
\[
\partial_tv_n+(-\Delta)^s v_n=\sigma_nb(x,t)\cdot Dv_n
\]
with $v_n(x,\omega)=\sigma_n v_\omega$. Set $w_n=v_n/\|v_n\|_{\H_2^1(Q_{(\omega,\tau)})}$. Then, we have
\[
\partial_tw_n+(-\Delta)^s w_n\leq |b||Dw_n|
\]
and $w_n$ is bounded in $ \H_2^1(Q_{(\omega,\tau)})$ (actually one has $\|w_n\|_{ \H_2^1(Q_{(\omega,\tau)})}=1$). Arguing as above, we have $|b||Dw_n|\in L^2(Q_{\omega,\tau})$, and hence $w_n\in\H_2^1$ and, in particular, the sequence $w_n$ is relatively compact in $\mathbb{H}^{1}_2(Q_{\omega,\tau})$ using the previous arguments. Therefore, there exists a subsequence to $w_n$ converging strongly to $w$ and by letting $n\to\infty$ we deduce
\[
\partial_tw+(-\Delta)^s w\leq |b||Dw|
\]
with $w(x,0)=0$. By the comparison principle, we deduce $w\leq0$. Similarly, using the same proof for $-w$ we conclude $w\geq0$ to get $w\equiv0$. Since $\|w_n\|_{\H^{1}_2(Q_{\omega,\tau})}=1$ and $w_n$ is strongly convergent we get a contradiction. We then conclude the existence of solutions by the Leray-Schauder fixed point theorem \cite[Theorem 11.6]{GT}. 
\end{proof}
\begin{rem}
The proof of the above result remains the same if one adds a forcing term $f\in L^2_{x,t}$.
\end{rem}
\begin{rem}
The existence of a suitable weak solution when $b\in L^\sP(Q_\tau)$, $\sP>\frac{d+2s}{2s-1}$, has been obtained in \cite[Example 3 p.335]{JS}. 
\end{rem}
\begin{rem}
Well-posedness of \eqref{drift} has been recently addressed in \cite{PPS}, see also \cite{AP,Ochoa} for the stationary problem, in the context of $L^1$ data for problems posed on bounded domains. In this case $v\in L^p(W^{1,p})$, $p<\frac{d+2s}{d+1}$ since for $b\in L^\sP$, $\sP>\frac{d+2s}{2s-1}$ one has $b\cdot Dv\in L^1_{x,t}$.
\end{rem}

%\bibliography{adjoint_frac}
%\bibliographystyle{abbrv}

\medskip

\begin{flushright}
\noindent \verb"alessandro.goffi@math.unipd.it"\\
Dipartimento di Matematica ``Tullio Levi-Civita'' \\
Universit\`a di Padova\\
via Trieste 63, 35121 Padova (Italy)
\end{flushright}

\end{document}